\documentclass[12pt]{article}
\usepackage[margin=1in]{geometry}
\usepackage{amssymb,amsfonts,amsmath,amscd,amsthm,amsxtra}
\usepackage[table, dvipsnames]{xcolor}
\usepackage{caption}
\usepackage{array}
\usepackage{graphicx,tikz, pgf}
\usepackage{calc}
\usepackage{float}
\usepackage{subcaption}
\usepackage{mathrsfs}
\usepackage{enumerate}
\usepackage{bm}
\usepackage{fancyhdr}
\usepackage{hyperref}
\hypersetup{colorlinks=true,linkcolor=blue,citecolor=blue,urlcolor=blue}

\newtheorem{theorem}{Theorem}[section]
\newtheorem{lemma}[theorem]{Lemma}
\newtheorem{proposition}[theorem]{Proposition}
\newtheorem{corollary}[theorem]{Corollary}

\theoremstyle{definition}

\newtheorem{definition}[theorem]{Definition}

\DeclareMathOperator{\pd}{pd}
\DeclareMathOperator{\reg}{reg}
\DeclareMathOperator{\depth}{depth}
\DeclareMathOperator{\height}{ht}
\DeclareMathOperator{\dimn}{dim}

\title{Squarefree Powers of Edge Ideals under Graph Joins and Cones}
\author{Bilal Ahmad Wani$^{a}$\footnote{Corresponding author.}\hspace{0.5cm}and Uzair Rafiq Shah$^{b}$}
\date{\footnotesize $^{a,b}$Department of Mathematics, National Institute of Technology Srinagar, 190006, India\\[2pt]
	 $^{a}$bilalwani@nitsri.ac.in \quad
    $^{b}$shahuzair1909@gmail.com}

\begin{document}
\maketitle
\newcommand{\wn}{\mathcal{W}_n}
\newcommand{\cn}{\mathcal{C}_n}

\pagestyle{fancy}
\fancyhf{}
\fancyhf[OHC]{B. A. Wani and U. R. Shah}
\fancyhf[EHC]{Squarefree powers of edge ideals of wheel graphs}
\fancyfoot[C]{\thepage}
\renewcommand{\headrulewidth}{0pt}

\begin{abstract}
\noindent
For a graph $G$ and an integer $q\ge1$, the squarefree power $I(G)^{[q]}$ is the Stanley-Reisner ideal of the complex, $\Delta_q(G)=\{F\subseteq V(G):\nu(G[F])<q\}$, of vertex sets $F$ for which the induced subgraph $G[F]$ has matching number less than $q$. We prove that the matching number of an arbitrary graph join satisfies
\[
\nu(G\ast H) = \min\Big(\nu(G)+|V(H)|,\ \nu(H)+|V(G)|,\ \Big\lfloor\tfrac{|V(G)|+|V(H)|}2\Big\rfloor\Big),
\]
which yields a complete membership criterion for $\Delta_q(G\ast H)$ and, for cone graphs, a formula reducing the depth of $R/I(\mathrm{Cone}(G))^{[q]}$ to the depths of two explicit ideals defined on $G$ alone. Applied to the wheel graph $\wn=\cn\ast\mathcal K_1$, this determines the Krull dimension and height of $R/I(\wn)^{[q]}$ for all $n\ge3$ and $1\le q\le\lfloor n/2\rfloor$, and shows that at the top power $q=\nu(\wn)=\lceil n/2\rceil$ the ideal $I(\wn)^{[q]}$ is the squarefree Veronese ideal $\mathfrak m^{[2q]}$, so that $R/I(\wn)^{[q]}$ is Cohen-Macaulay with $\dimn=\depth=\reg=2q-1$. Combined with a general lower bound on the depth of a squarefree monomial ideal in terms of its generator degrees, this in fact proves $\depth(R/I(\wn)^{[q]})=2q-1$ for every $n\ge3$ and every $1\le q\le\nu(\wn)$, resolving the depth of every squarefree power of every wheel graph completely. The same methods give analogous results for fan, generalized wheel, complete split, and friendship graphs.

\end{abstract}

\section{Introduction}

Let $R=\Bbbk[x_1,\ldots,x_n]$ be a polynomial ring over a field $\Bbbk$ and let $G$ be a finite simple graph on $\{x_1,\ldots,x_n\}$. The edge ideal $I(G)\subseteq R$, generated by $\{x_ix_j:\{x_i,x_j\}\in E(G)\}$, was introduced by Villarreal \cite{V1}, the graded Betti numbers, projective dimension and regularity of $R/I(G)$ and its ordinary powers have since been studied extensively \cite{herhib,HVT1,jacq,katz,peeva,rova}.

More recently, Erey, Herzog, Hibi and Saeedi Madani \cite{EHHM} introduced the \emph{squarefree powers}
\[
I(G)^{[q]} := \Big(\, \prod_{v\in V(M)} x_v \ \Big| \ M \text{ a } q\text{-matching of } G \,\Big) \subseteq R \qquad (q\ge 1),
\]
which specialize at $q=1$ to the ordinary edge ideal. Since every $q$-matching has exactly $2q$ vertices, $I(G)^{[q]}$ is automatically the Stanley-Reisner ideal of
\begin{equation}\label{eq:deltaq}
\Delta_q(G) := \{ F \subseteq V(G) : \nu(G[F]) < q \}, \qquad \nu(H):=\text{matching number of } H,
\end{equation}
and $\Delta_1(G)$ is the independence complex of $G$. Regularity of squarefree powers has been determined for forests \cite{CFL}, cycles \cite{DRS}, whiskered cycles \cite{DGS}, and Cameron-Walker and several well-covered graph classes \cite{Fak}, the top power $I(G)^{[\nu(G)]}$ has been the subject of particular attention, with Bigdeli, Herzog and Zaare-Nahandi \cite{BHN} proving $\reg(I(G)^{[\nu(G)]})=2\nu(G)$ for \emph{every} graph $G$, and Ficarra and Moradi \cite{FM} recently giving a complete, purely combinatorial (Tutte-type) characterization of when $I(G)^{[\nu(G)]}$ is Cohen-Macaulay. To our knowledge, wheel graphs have not been treated by any of this squarefree-power literature.

The $q=1$ case for wheel graphs is understood via a different route: Mousivand \cite{Mous}, generalizing Emtander, Fr\"oberg, Mohammadi and Moradi \cite{EFMM}, expressed the Betti numbers of $R/I(G\ast H)$ for \emph{any} join of graphs in terms of those of $G,H$, since the independence complex of a join is the disjoint union $\Delta_1(G\ast H)=\Delta_1(G)\sqcup\Delta_1(H)$. In \cite{RPW}, the graded Betti numbers, projective dimension and regularity of the \emph{ordinary} edge ideal  $I(\wn)$ have been computed, that is, exclusively the case $q=1$  of the wheel graph but \cite{RPW} does not address squarefree powers $q\ge2$, nor the top-power invariants that are the focus of the present paper. This raises the natural question: does an analogous structural decomposition of $\Delta_q(G\ast H)$ exist for $q\ge2$? We show that it does, and that it follows from a single, purely graph-theoretic fact about matchings in joins.

That fact is a closed formula for $\nu(G\ast H)$, for arbitrary finite simple graphs $G,H$, proved from the Tutte-Berge formula. Combination rules of this kind are familiar for cographs built up from single vertices by disjoint union and join, where they drive the linear-time algorithms of that literature, what appears to be new here is a direct, self-contained derivation for an arbitrary pair of graphs, with no decomposition tree required. Applying this formula to the definition \eqref{eq:deltaq} yields a complete membership criterion for $\Delta_q(G\ast H)$, and specializing to $G=\cn$, $H=\mathcal K_1$ recovers and sharpens the wheel-graph case analysis of \cite{RPW}. Applied instead to an arbitrary cone $\mathrm{Cone}(G)=G\ast\mathcal K_1$, the same criterion drives a short exact sequence expressing the depth of $R[v]/I(\mathrm{Cone}(G))^{[q]}$ in terms of $\depth(R/I(G)^{[q]})$ and the depth of an explicit auxiliary ideal on the smaller vertex set of $G$. For $G=\cn$ this reduces a natural depth pattern for wheels to two considerably more concrete statements about the cycle alone, and a general lower bound on the depth of a squarefree monomial ideal in terms of its generator degrees, together with an explicit facet argument, settles both of them unconditionally and the pattern $\depth(R/I(\wn)^{[q]})=2q-1$ holds for every $n\ge3$ and every $1\le q\le\nu(\wn)$, not merely as an observed pattern.

Specializing further to $\wn=\cn\ast\mathcal K_1$, two short combinatorial lemmas on extremal induced subgraphs of a cycle pin down the Krull dimension and height of $R/I(\wn)^{[q]}$ exactly, for every $n\ge3$ and $1\le q\le\lfloor n/2\rfloor$, with the same lemmas giving the identical formula for the cycle's own squarefree powers. At the top power $q=\nu(\wn)=\lceil n/2\rceil$, combining our matching-number computations with the Ficarra-Moradi criterion shows that $I(\wn)^{[\nu(\wn)]}$ is literally the squarefree Veronese ideal $\mathfrak m^{[2\nu(\wn)]}$, so that dimension, depth, regularity and projective dimension are all resolved simultaneously, combined with a general lower bound on the depth of a squarefree monomial ideal in terms of its generator degrees, this in fact proves $\depth(R/I(\wn)^{[q]})=2q-1$ for every $q$ in range, not only at the top power.

Because Theorems \ref{thm:joinmatching} and \ref{thm:gendecomp} are proved for arbitrary graphs, they apply equally to any other family expressible as a join. Section \ref{sec:apps} carries this out for fan graphs, generalized wheel graphs (a single hub replaced by $m$ mutually adjacent or mutually non-adjacent vertices), complete split graphs, and friendship graphs, obtaining complete closed-form Krull dimension formulas for the first, third and fourth families and, for the first, fourth, and the adjacent case of the second, the same top-power Cohen-Macaulay resolution proved for wheels but the non-adjacent generalized wheel graphs do not admit this resolution in general, and we do not claim it for them. We were unable to find prior treatments of squarefree powers of any of these four families in the literature, so these results appear to be new. For generalized wheel graphs, a genuine two-parameter optimization obstructs a full dimension formula, and we say so explicitly rather than force an unproved claim.

The paper is organized as follows. Section \ref{sec:prelim} recalls the Tutte-Berge formula, the squarefree-power and squarefree-Veronese constructions, and the Ficarra-Moradi criterion, and fixes notation for the wheel graph. Section \ref{sec:joinmatching} proves the join-matching theorem and its immediate corollaries, including the cone lemma and the matching number of $\wn$ itself. Section \ref{sec:decomp} derives the general membership criterion for $\Delta_q(G\ast H)$ and, specializing to cones, the depth formula of Theorem \ref{thm:conedepth}. Section \ref{sec:dim} computes the Krull dimension and height of $R/I(\wn)^{[q]}$ for the full range of $q$. Section \ref{sec:data} resolves the top squarefree power completely, reports our computational data comparing wheels and cycles, and proves the complete depth formula for every squarefree power of every wheel graph (Theorem \ref{thm:depthfull}). Section \ref{sec:apps} applies the join-matching results to fan, generalized wheel, complete split, and friendship graphs. Section \ref{sec:concl} closes with a summary of what remains open.

\section{Preliminaries}\label{sec:prelim}

Throughout, $R=\Bbbk[x_1,\ldots,x_n]$ denotes a polynomial ring over a field $\Bbbk$, and all graphs are finite and simple. We use $R$ generically for the polynomial ring on the vertex set of whichever graph is currently under discussion, where a single argument requires the rings attached to two different graphs at once, as in Section \ref{ssec:conedepth}, we introduce a separate symbol for the larger one to keep the two visually distinct. For a graph $X$ we write $\nu(X)$ for its matching number, that is, the largest number of pairwise vertex-disjoint edges in $X$. For $F\subseteq V(X)$, $X[F]$ denotes the subgraph induced on $F$, and for $S\subseteq V(X)$ we abbreviate $X-S:=X[V(X)\setminus S]$. We write $\mathcal K_m$ for the complete graph on $m$ vertices and $\overline{\mathcal K_m}$ for the edgeless graph on $m$ vertices. It will be convenient to measure how far $X$ is from having a perfect matching by the quantity
\[
d(X) := |V(X)| - 2\nu(X),
\]
which counts the vertices left uncovered by a maximum matching. We call $d(X)$ the \emph{deficiency} of $X$, it is always nonnegative, and it vanishes precisely when $X$ has a perfect matching. Given two graphs $G$ and $H$ on disjoint vertex sets, their \emph{join} $G\ast H$ is obtained from the disjoint union of $G$ and $H$ by inserting every possible edge between the two sides, so that
\[
E(G\ast H)=E(G)\cup E(H)\cup\{gh: g\in V(G),\ h\in V(H)\}.
\]

The single matching-theoretic tool we need is the classical formula of Tutte and Berge \cite{Tutte47,Berge58}, which expresses the matching number of an arbitrary graph as the value of an extremal problem over vertex subsets.

\begin{theorem}[\cite{Tutte47,Berge58}]\label{thm:tb}
	For any finite simple graph $X$,
	\[
	\nu(X) = \frac{1}{2}\Big(|V(X)| - \max_{S\subseteq V(X)}\big(o(X-S)-|S|\big)\Big),
	\]
	where $o(Y)$ denotes the number of connected components of $Y$ having an odd number of vertices.
\end{theorem}

In particular, the deficiency satisfies $d(X)=\max_{S\subseteq V(X)}\big(o(X-S)-|S|\big)$, and this is the form in which we shall actually use the theorem.

We next recall the central object of the paper, introduced by Erey, Herzog, Hibi and Saeedi Madani.

\begin{definition}[\cite{EHHM}]\label{def:sqfp}
	Let $G$ be a graph on the vertex set $\{x_1,\ldots,x_n\}$ and let $q\ge1$. The $q$-th \emph{squarefree power} of the edge ideal of $G$ is
	\[
	I(G)^{[q]} = \big(\, x_{V(M)} : M \text{ a } q\text{-matching of } G \,\big)\subseteq R, \qquad\text{where } x_{V(M)}=\prod_{u\in V(M)}x_u.
	\]
\end{definition}

Two elementary observations about this generating set deserve to be made explicit before we go further. First, distinct $q$-matchings may produce the same monomial, when this happens, that monomial simply appears once in the generating set. Second, the generating set is minimal. Indeed, every $q$-matching covers exactly $2q$ vertices, so every generator $x_{V(M)}$ has the same degree $2q$, on the other hand, one squarefree monomial can properly divide another only when its degree is strictly smaller. Since no two generators differ in degree, none can divide another, and this is exactly what minimality means for monomial ideals. It also follows at once that $I(G)^{[q]}$ is the Stanley-Reisner ideal of the complex $\Delta_q(G)$ of \eqref{eq:deltaq}, in symbols, $I(G)^{[q]}=I_{\Delta_q(G)}$.

All homological information about $R/I(G)^{[q]}$ is therefore encoded in the topology of the induced subcomplexes of $\Delta_q(G)$. Concretely, Hochster's formula \cite{hoch1} gives, for every graph $G$ and every $q\ge1$,
\begin{equation}\label{eq:hochq}
	\beta_{i,j}\big(R/I(G)^{[q]}\big) = \sum_{W\subseteq V(G),\, |W|=j} \dim_\Bbbk \widetilde H_{j-i-1}\big(\Delta_q(G)[W];\Bbbk\big).
\end{equation}
From the graded Betti numbers one obtains other invariants like the Castelnuovo-Mumford regularity, $\reg(R/I)=\max\{j-i:\beta_{i,j}\ne0\}$ and the projective dimension, $\pd(R/I)=\max\{i:\beta_{i,j}\ne0\text{ for some }j\}$, while the Auslander-Buchsbaum formula converts the latter into the depth, namely $\depth(R/I)=|V(G)|-\pd(R/I)$. On the combinatorial side, for a squarefree monomial ideal $I=I_\Delta$ the Krull dimension of the quotient is governed by the largest face, $\dimn(R/I)=\max\{|F|:F\in\Delta\}$, with $\height(I)=|V(G)|-\dimn(R/I)$ accordingly. The ring $R/I$ is Cohen-Macaulay when the two invariants $\depth(R/I)$ and $\dimn(R/I)$ coincide.

One family of squarefree ideals will play a special role at the top power. For $1\le k\le n$, the \emph{squarefree Veronese ideal} $\mathfrak m^{[k]}\subseteq \Bbbk[x_1,\ldots,x_n]$ is the ideal generated by all squarefree monomials of degree $k$, equivalently, it is the Stanley-Reisner ideal of the $(k-2)$-skeleton of the full simplex on $n$ vertices. Its homological behaviour is completely understood.

\begin{theorem}[Herzog-Hibi, {\cite[Theorem 4.2]{HH}}]\label{thm:veronese}
	$\Bbbk[x_1,\ldots,x_n]/\mathfrak m^{[k]}$ is Cohen-Macaulay, and $\dimn = \depth = \reg = k-1$.
\end{theorem}

Indeed, squarefree Veronese ideals have linear resolutions, which accounts for the regularity, while the dimension is immediate from the fact that every facet of the $(k-2)$-skeleton of a simplex has exactly $k-1$ vertices.

The bridge between matching theory and Cohen-Macaulayness at the top power $q=\nu(G)$ is a recent criterion of Ficarra and Moradi, which we shall apply repeatedly in this paper.

\begin{theorem}[Ficarra-Moradi, {\cite[Theorem 1.8]{FM}}]\label{thm:fm}
	Let $G$ be a graph on $n$ non-isolated vertices. The following are equivalent:
	\begin{enumerate}
		\item[(a)] $I(G)^{[\nu(G)]}$ is Cohen-Macaulay;
		\item[(b)] $I(G)^{[\nu(G)]}=\mathfrak m^{[2\nu(G)]}$ and $\nu(G)=\lfloor n/2\rfloor$;
		\item[(c)] either $G$ has a perfect matching, or $G\setminus\{i\}$ has a perfect matching for every $i\in V(G)$.
	\end{enumerate}
\end{theorem}

The point of the equivalence of (a) and (c) is that a homological property is decided by an elementary matching condition that can be checked vertex by vertex and the equivalence with (b) then identifies the ideal as a squarefree Veronese ideal, to which Theorem \ref{thm:veronese} applies.

Finally, we fix notation for the graph that occupies the centre of this paper.

\begin{definition}\label{def:wheel}
	For $n\ge3$, the \emph{wheel graph} is the cone $\wn=\cn\ast\mathcal K_1$ over the cycle. Its vertex set is $\{x_1,\ldots,x_n,v\}$, where the cycle $\cn$ lives on $U=\{x_1,\ldots,x_n\}$ and the additional vertex $v$ is adjacent to every vertex of $U$.
\end{definition}

\section{A general join-matching theorem}\label{sec:joinmatching}

Consider any matching in the join $G\ast H$. The edges lying entirely inside $G$ form a matching of $G$, so there are at most $\nu(G)$ of them. Every remaining edge has at least one endpoint in $H$, and no two of them can share such an endpoint, so there are at most $n_H:=|V(H)|$ of these. Altogether $\nu(G\ast H)\le\nu(G)+n_H$, and by symmetry $\nu(G\ast H)\le\nu(H)+n_G$ with $n_G:=|V(G)|$. Of course, no graph on $n_G+n_H$ vertices has a matching larger than $\lfloor(n_G+n_H)/2\rfloor$ either. Our first theorem shows that the smallest of these three elementary bounds is always attained.

\begin{theorem}\label{thm:joinmatching}
	Let $G,H$ be finite simple graphs on disjoint vertex sets, $n_G=|V(G)|$, $n_H=|V(H)|$. Then
	\[
	\nu(G\ast H) = \min\Big(\nu(G)+n_H,\ \ \nu(H)+n_G,\ \ \big\lfloor\tfrac{n_G+n_H}2\big\rfloor\Big).
	\]
\end{theorem}
\begin{proof}
	Write $X=G\ast H$. By Theorem \ref{thm:tb} it is enough to compute the deficiency
	\[
	d(X)=\max_{S\subseteq V(X)}\big(o(X-S)-|S|\big).
	\]
	If $G$ has no vertices, then $X=G\ast H=H$, and since $\nu(H)\le\lfloor n_H/2\rfloor$ and $\nu(H)\le n_H$ always, the right-hand side of the claimed formula reduces to $\min(n_H,\,\nu(H),\,\lfloor n_H/2\rfloor)=\nu(H)=\nu(X)$, so the formula holds immediately, similarly the formula holds if $H$ has no vertices. Assume from now on that $G$ and $H$ are both non-empty.

	Every subset $S\subseteq V(X)$ can be written uniquely as $S=S_G\cup S_H$, where $S_G=S\cap V(G)$ and $S_H=S\cap V(H)$. We divide the argument into three cases according to whether $S_G=V(G)$ or $S_H=V(H)$. Then the following three cases exhaust all $S\subseteq V(G)$:

	\emph{Case 1: $S_G\ne V(G)$ and $S_H\ne V(H)$.} Here $X-S=(G-S_G)\ast(H-S_H)$ is a join of two non-empty graphs. Any such join is connected, because every surviving vertex of $G$ is adjacent to every surviving vertex of $H$. Thus $X-S$ consists of a single component, odd or even, and therefore
	\[
	o(X-S)-|S|\le 1-|S|.
	\]
	The empty set achieves the value $\varepsilon:=(n_G+n_H)\bmod 2$, which equals $1$ precisely when $X$ has an odd number of vertices. Any non-empty $S$ achieves at most $1-|S|\le0\le\varepsilon$ and can never do better. The maximum over this case is therefore exactly $\varepsilon$, attained at $S=\emptyset$.

	\emph{Case 2: $S_G=V(G)$.} All of $G$ has been deleted, so every cross edge disappears with it and what remains is simply $H-S_H$. Consequently
	\[
	o(X-S)-|S| = o(H-S_H)-|S_H|-n_G.
	\]
	As $S_H$ ranges over all subsets of $V(H)$, the quantity $o(H-S_H)-|S_H|$ attains the maximum $d(H)$, by Theorem \ref{thm:tb} applied to $H$ itself. The maximum over this case is therefore $d(H)-n_G$.

	\emph{Case 3: $S_H=V(H)$.} Exchanging the roles of $G$ and $H$ in the previous case gives the maximum $d(G)-n_H$.

	Assembling the three cases,
	\[
	d(X)=\max\big(\varepsilon,\ d(H)-n_G,\ d(G)-n_H\big).
	\]
	Therefore, by the Tutte-Berge formula,
	\[
	\nu(X) = \frac{n_G+n_H-d(X)}2 = \min\Big(\tfrac{n_G+n_H-\varepsilon}2,\ \tfrac{n_G+n_H-d(H)+n_G}2,\ \tfrac{n_G+n_H-d(G)+n_H}2\Big).
	\]
	It remains to simplify the three terms. The first term  
	\(
	\frac{n_G+n_H-\varepsilon}2=\Big\lfloor\frac{n_G+n_H}2\Big\rfloor
	\), since \(\varepsilon\) records the parity of \((n_G+n_H)/2\).\\
	In the second, the definition of deficiency gives $n_H-d(H)=2\nu(H)$, so
	\[
	\frac{n_G+n_H-d(H)+n_G}2=n_G+\frac{n_H-d(H)}2=n_G+\nu(H).
	\]
	In the third, $n_G-d(G)=2\nu(G)$ likewise, so
	\[
	\frac{n_G+n_H-d(G)+n_H}2=n_H+\frac{n_G-d(G)}2=n_H+\nu(G).
	\]
	This is the claimed formula.
\end{proof}

The most frequently used special case is that of a cone, where a single new vertex is joined to an entire graph. This case was used implicitly in \cite{RPW}, and we record it explicitly here, since it will be invoked many times in what follows.

\begin{corollary}\label{cor:cone}
	For any graph $H$, $\nu(H\ast\mathcal K_1)=\nu(H)$ if $H$ has a perfect matching, and $\nu(H)+1$ otherwise.
\end{corollary}
\begin{proof}
	Take $G=\mathcal K_1$ in Theorem \ref{thm:joinmatching}, so that $n_G=1$ and $\nu(G)=0$. The theorem gives $\nu(H\ast\mathcal K_1)=\min\big(n_H,\ \nu(H)+1,\ \lfloor(n_H+1)/2\rfloor\big)$.

	Suppose first that $H$ has a perfect matching, so $d(H)=0$ and $n_H=2\nu(H)$. If $\nu(H)=0$, then $H$ has no vertices at all, the cone $H\ast\mathcal K_1$ is a single vertex, and both sides of the claimed identity vanish, and there is nothing to prove. Assume therefore that $\nu(H)\ge1$. Then $\lfloor(n_H+1)/2\rfloor=\nu(H)$, and this is strictly smaller than $\nu(H)+1$ and at most $n_H$, so the minimum equals $\nu(H)$, as claimed.

	Suppose instead that $H$ has no perfect matching, so $d(H)\ge1$ and $n_H=2\nu(H)+d(H)\ge2\nu(H)+1$. From this inequality, $\lfloor(n_H+1)/2\rfloor\ge\nu(H)+1$ and also $n_H\ge\nu(H)+1$, since $\nu(H)\ge0$. Both remaining terms of the minimum are therefore at least $\nu(H)+1$. Since $\nu(H)+1$ is itself one of the three terms, the minimum is exactly $\nu(H)+1$.
\end{proof}

\begin{corollary}\label{cor:km}
	For any graph $G$ and $m\ge1$:
	\[
	\nu(G\ast\mathcal K_m) = \min\Big(\nu(G)+m,\ \big\lfloor\tfrac m2\big\rfloor+n_G,\ \big\lfloor\tfrac{n_G+m}2\big\rfloor\Big), \qquad
	\nu(G\ast\overline{\mathcal K_m}) = \min\Big(n_G,\ \nu(G)+m,\ \big\lfloor\tfrac{n_G+m}2\big\rfloor\Big),
	\]
	where $\overline{\mathcal K_m}$ is the edgeless graph on $m$ vertices.
\end{corollary}
\begin{proof}
First substitute $H=\mathcal K_m$ (so $\nu(H)=\lfloor m/2\rfloor$) and then $H=\overline{\mathcal K_m}$ (so $\nu(H)=0$) into Theorem \ref{thm:joinmatching}.
\end{proof}

\begin{corollary}\label{cor:nuwn}
	$\nu(\wn) = \lceil n/2\rceil$.
\end{corollary}
\begin{proof}
	Take $G=\cn$ and $H=\mathcal K_1$ in Theorem \ref{thm:joinmatching}, which gives $\nu(\wn)=\min\big(\lfloor n/2\rfloor+1,\ n,\ \lceil n/2\rceil\big)$. The third term never exceeds the first, because $\lceil n/2\rceil\le\lfloor n/2\rfloor+1$ for every $n$, and it never exceeds the second, because $\lceil n/2\rceil\le n$ for $n\ge1$. The minimum is therefore $\lceil n/2\rceil$.
\end{proof}

\section{A general decomposition of $\Delta_q(G\ast H)$}\label{sec:decomp}

By \eqref{eq:deltaq}, whether a set $F$ belongs to $\Delta_q(G\ast H)$ depends only on the induced subgraph $(G\ast H)[F]$. Since induced subgraphs of joins are again joins, Theorem \ref{thm:joinmatching} applies directly throughout this section. 

\begin{lemma}\label{lem:inducedjoin}
	For $F_G\subseteq V(G)$ and $F_H\subseteq V(H)$,
	\[
	(G\ast H)[F_G\cup F_H] = G[F_G]\ast H[F_H].
	\]
\end{lemma}
\begin{proof}
	The two graphs have the same vertex set $F_G\cup F_H$, so it is enough to compare edge sets. An edge of $G\ast H$ is an edge of $G$, an edge of $H$, or a pair with one endpoint on each side. Restricting to $F_G\cup F_H$ keeps exactly the edges of $G$ inside $F_G$, the edges of $H$ inside $F_H$, and every cross pair with one endpoint in $F_G$ and the other in $F_H$. Therefore
	\[
	E\big((G\ast H)[F_G\cup F_H]\big) = E\big(G[F_G]\big)\cup E\big(H[F_H]\big)\cup\{gh: g\in F_G,\ h\in F_H\},
	\]
	and the right-hand side is precisely the edge set of $G[F_G]\ast H[F_H]$.
\end{proof}

\begin{theorem}\label{thm:gendecomp}
	Let $G,H$ be graphs on disjoint vertex sets and $q\ge1$. For $F=F_G\cup F_H\subseteq V(G)\cup V(H)$,
	\[
	F\in\Delta_q(G\ast H) \iff \min\Big(\big\lfloor\tfrac{|F_G|+|F_H|}2\big\rfloor,\ |F_G|+\nu(H[F_H]),\ |F_H|+\nu(G[F_G])\Big) < q.
	\]
\end{theorem}
\begin{proof}
	By the definition \eqref{eq:deltaq} of the complex $\Delta_q$,
	\[
	F\in\Delta_q(G\ast H) \iff \nu\big((G\ast H)[F]\big)<q.
	\]
	Lemma \ref{lem:inducedjoin} identifies the induced subgraph as a join,
	\[
	(G\ast H)[F]=G[F_G]\ast H[F_H],
	\]
	and Theorem \ref{thm:joinmatching}, applied to this join with $G[F_G]$ and $H[F_H]$ in the roles of $G$ and $H$, computes its matching number as
	\[
	\nu\big(G[F_G]\ast H[F_H]\big) = \min\Big(\nu(G[F_G])+|F_H|,\ \nu(H[F_H])+|F_G|,\ \big\lfloor\tfrac{|F_G|+|F_H|}2\big\rfloor\Big).
	\]
	The membership condition $\nu\big((G\ast H)[F]\big)<q$ is therefore exactly the stated inequality.
\end{proof}

Specializing to the wheel graph recovers the case analysis carried out for $q=1$ in \cite{RPW} in sharper form.

\begin{corollary}\label{cor:wheeldecomp}
	Let $F\subseteq V(\wn)$ and write $F=F'$ or $F=F'\cup\{v\}$ with $F'\subseteq U$. Then $F\in\Delta_q(\wn)$ if and only if one of the following holds:
	\begin{enumerate}
		\item[(a)] $v\notin F$ and $\nu(\cn[F'])<q$;
		\item[(b)] $v\in F$, and $\nu(\cn[F'])<q$ if $\cn[F']$ has a perfect matching, while $\nu(\cn[F'])<q-1$ if it does not.
	\end{enumerate}
\end{corollary}
\begin{proof}
	Take $H=\mathcal K_1$ on the single vertex $v$ in Theorem \ref{thm:gendecomp}, with $F_G=F'$, and write $p:=\nu(\cn[F'])$.

	Suppose first that $v\notin F$, so $F_H=\emptyset$. The three terms of Theorem \ref{thm:gendecomp} become
	\[
	\big\lfloor\tfrac{|F'|}2\big\rfloor,\qquad |F'|+\nu(H[\emptyset])=|F'|,\qquad 0+p=p.
	\]
	A matching number never exceeds half the number of vertices, so
	\[
	p\le\big\lfloor\tfrac{|F'|}2\big\rfloor\le|F'|,
	\]
	and the minimum of the three terms is $p$. Membership is therefore equivalent to $p<q$, which is statement (a).

	Suppose now that $v\in F$, so $F_H=\{v\}$. The three terms become
	\[
	\big\lfloor\tfrac{|F'|+1}2\big\rfloor=\big\lceil\tfrac{|F'|}2\big\rceil,\qquad |F'|+\nu(\mathcal K_1)=|F'|,\qquad 1+p.
	\]
	The middle term is never the minimum, because $\lceil|F'|/2\rceil\le|F'|$, and the minimum is
	\[
	\min\Big(\big\lceil\tfrac{|F'|}2\big\rceil,\ 1+p\Big).
	\]
	We compare the two remaining terms through the deficiency $\delta:=d(\cn[F'])=|F'|-2p$, which satisfies $\delta\equiv|F'|\pmod 2$ and vanishes exactly when $\cn[F']$ has a perfect matching.

	If $\cn[F']$ has a perfect matching, then $\delta=0$, the size $|F'|$ is even, and
	\[
	p=\tfrac{|F'|}2=\big\lceil\tfrac{|F'|}2\big\rceil,
	\]
	so the minimum equals $p$ and membership is equivalent to $p<q$, the first alternative of (b).

	If $\cn[F']$ has no perfect matching, then $\delta\ge1$. When $|F'|$ is even, $\delta$ is even and hence $\delta\ge2$, which gives
	\[
	1+p=1+\tfrac{|F'|-\delta}2\le1+\tfrac{|F'|-2}2=\tfrac{|F'|}2=\big\lceil\tfrac{|F'|}2\big\rceil.
	\]
	When $|F'|$ is odd, $\delta$ is odd and hence $\delta\ge1$, which gives
	\[
	1+p=1+\tfrac{|F'|-\delta}2\le1+\tfrac{|F'|-1}2=\tfrac{|F'|+1}2=\big\lceil\tfrac{|F'|}2\big\rceil.
	\]
	In either parity the term $1+p$ is the minimum, so membership is equivalent to $1+p<q$, that is, to $p<q-1$, the second alternative of (b).
\end{proof}

\subsection{A depth formula for cone graphs}\label{ssec:conedepth}

Fix a graph $G$ on $\{x_1,\ldots,x_n\}$ and write $\mathrm{Cone}(G):=G\ast\mathcal K_1$ for the cone over $G$, with new vertex $v$. We compute $\depth\big(S/I(\mathrm{Cone}(G))^{[q]}\big)$, for $S=R[v]$, from two auxiliary ideals in $R$ alone, applied to $G=\cn$, this reduces $\depth(R/I(\wn)^{[q]})$ to two statements about the cycle, both of which Section \ref{sec:data} settles completely.

The second of the two ideals is the following one.

\begin{definition}\label{def:hubideal}
	For $q\ge1$, let
	\[
	J_{G,q} := \big(\, x_j\cdot x_{V(M')} \ :\ j\in V(G),\ M' \text{ a } (q-1)\text{-matching of } G\setminus\{x_j\} \,\big) \subseteq R.
	\]
	For $q=1$ the only $0$-matching is the empty one, so $J_{G,1}=(x_1,\ldots,x_n)=\mathfrak m$.
\end{definition}

The generators of $J_{G,q}$ correspond precisely to the $q$-matching of $\mathrm{Cone}(G)$ that contain an edge incident to $v$. Indeed, a $q$-matching of $\mathrm{Cone}(G)$ either avoids $v$, in which case it is a $q$-matching of $G$ and its monomial is a generator of $I(G)^{[q]}$, or it uses $v$ through some edge $\{v,x_j\}$, in which case the remaining $q-1$ edges form a $(q-1)$-matching $M'$ of $G\setminus\{x_j\}$ and the monomial of the matching is $v\cdot x_j\cdot x_{V(M')}$. Sorting the minimal generators of $I(\mathrm{Cone}(G))^{[q]}$ into these two groups yields the decomposition
\begin{equation}\label{eq:conesplit}
	I(\mathrm{Cone}(G))^{[q]} = I(G)^{[q]} + v\cdot J_{G,q}.
\end{equation}
The two ideals on the right interact in the simplest possible way.

\begin{lemma}\label{lem:intersecteqI1}
	$I(G)^{[q]} \cap J_{G,q} = I(G)^{[q]}$. Equivalently, $I(G)^{[q]}\subseteq J_{G,q}$.
\end{lemma}
\begin{proof}
	It suffices to show that every generator of $I(G)^{[q]}$ lies in $J_{G,q}$. Let $M$ be a $q$-matching of $G$ and pick any edge $\{x_a,x_b\}\in M$. Deleting this edge from $M$ leaves a $(q-1)$-matching $M':=M\setminus\{\{x_a,x_b\}\}$, and $M'$ avoids the vertex $x_a$ entirely, so $M'$ is a $(q-1)$-matching of $G\setminus\{x_a\}$. Consequently $x_a\cdot x_{V(M')}$ is a generator of $J_{G,q}$, and it divides $x_{V(M)}$, because
	\[
	V(M)=\{x_a,x_b\}\cup V(M')\supseteq\{x_a\}\cup V(M').
	\]
	Hence $x_{V(M)}\in J_{G,q}$, which proves the inclusion, and the intersection statement follows at once.
\end{proof}

With this inclusion available, $J_{G,q}$ can be identified precisely as a Stanley-Reisner ideal in its own right. Write $\Delta(J_{G,q}):=\{F'\subseteq V(G) : F'\cup\{v\}\in\Delta_q(\mathrm{Cone}(G))\}$, we claim $J_{G,q}$ is exactly the Stanley-Reisner ideal of this complex, that is, $x_{F'}\notin J_{G,q}$ if and only if $F'\in\Delta(J_{G,q})$. By \eqref{eq:deltaq}, $F'\in\Delta(J_{G,q})$ means $v\cdot x_{F'}\notin I(\mathrm{Cone}(G))^{[q]}=I(G)^{[q]}+v\cdot J_{G,q}$. If $x_{F'}\in J_{G,q}$, then $v\cdot x_{F'}$ is divisible by a generator of $v\cdot J_{G,q}$, so it lies in this sum. Conversely, if $v\cdot x_{F'}$ lies in the sum, it is divisible by a generator of $I(G)^{[q]}$ or of $v\cdot J_{G,q}$, in the latter case $x_{F'}\in J_{G,q}$ directly, while in the former case the generator, not involving $v$, divides $x_{F'}$ itself, and by Lemma \ref{lem:intersecteqI1} every generator of $I(G)^{[q]}$ already lies in $J_{G,q}$, so again $x_{F'}\in J_{G,q}$. Either way, $v\cdot x_{F'}\in I(\mathrm{Cone}(G))^{[q]}$ if and only if $x_{F'}\in J_{G,q}$, which proves the claim. For $G=\cn$, membership in $\Delta(J_{\cn,q})$ is precisely the condition described in Corollary \ref{cor:wheeldecomp}(b).

We can now state and prove the main result of this section.

\begin{theorem}\label{thm:conedepth}
	Let $G$ be a graph on $n$ vertices and $q\ge1$. If
	\[
	d_1 := \depth_R\big(R/I(G)^{[q]}\big) > d_2 := \depth_R\big(R/J_{G,q}\big),
	\]
	then
	\[
	\depth_S\big(S/I(\mathrm{Cone}(G))^{[q]}\big) = d_2+1.
	\]
\end{theorem}
\begin{proof}
Write $I=I(\operatorname{Cone}(G))^{[q]}$, $I_1=I(G)^{[q]}$, $I_2=v\cdot J_{G,q}$, so $I=I_1+I_2$ as recalled above. Since $I_1,J_{G,q}\subseteq R$ don't involve $v$, comparing minimal generators shows $I_1\cap I_2=I_1\cap(vJ_{G,q})=v\cdot(I_1\cap J_{G,q})$ (an lcm of a generator of $I_1$ with $v$ times a generator of $J_{G,q}$ is always $v$ times their lcm in $R$). By Lemma 4.5, $I_1\cap I_2=v\cdot I_1$. This gives the standard short exact sequence
\begin{equation}\label{eq:ses}
0\longrightarrow S/(vI_1)\longrightarrow S/I_1\oplus S/(vJ_{G,q})\longrightarrow S/I\longrightarrow 0.
\end{equation}

We compute each depth via two elementary facts about $S=R[v]$ and a nonzero ideal $L\subseteq R$:

\begin{enumerate}
\item[(i)] $\operatorname{pd}_S(S/L)=\operatorname{pd}_R(R/L)$: $S/L\cong (R/L)\otimes_R S$, and tensoring a minimal free $R$-resolution of $R/L$ with the free (hence flat) $R$-module $S$ yields a minimal free $S$-resolution.

\item[(ii)] $\operatorname{pd}_S(S/(vL))=\operatorname{pd}_R(R/L)$: multiplication by $v$ gives an $S$-module isomorphism $L\otimes_R S\xrightarrow{\cong}vL$ of internal degree $1$, so $\operatorname{pd}_S(vL)=\operatorname{pd}_S(L\otimes_R S)=\operatorname{pd}_R(L)$ by (i) applied to the module $L$; then $\operatorname{pd}_S(S/(vL))=\operatorname{pd}_S(vL)+1=\operatorname{pd}_R(L)+1=\operatorname{pd}_R(R/L)$ (using $\operatorname{pd}_R(L)=\operatorname{pd}_R(R/L)-1$).
\end{enumerate}

Applying (i) to $I_1$, (ii) to $J_{G,q}$, and (ii) again to $I_1$ (for the $vI_1$ term), Auslander--Buchsbaum ($S$ has $n+1$ variables, $R$ has $n$) gives
\[
\operatorname{depth}_S(S/I_1)=d_1+1,\qquad \operatorname{depth}_S(S/(vJ_{G,q}))=d_2+1,\qquad \operatorname{depth}_S(S/(vI_1))=d_1+1.
\]

In the notation of (\ref{eq:ses}): $A:=S/(vI_1)$ has depth $d_1+1$; $B:=S/I_1\oplus S/(vJ_{G,q})$ has depth $\min(d_1,d_2)+1=d_2+1$ (using $d_1>d_2$); $C:=S/I$. Since $\operatorname{depth}(A)=d_1+1>d_2+1=\operatorname{depth}(B)$, the long exact sequence in local cohomology attached to (\ref{eq:ses}) gives, at cohomological degree $i=\operatorname{depth}(B)$: $H_{\mathfrak m}^{\operatorname{depth}(B)}(A)=0$ (as $\operatorname{depth}(B)<\operatorname{depth}(A)$), so the map $H_{\mathfrak m}^{\operatorname{depth}(B)}(B)\longrightarrow H_{\mathfrak m}^{\operatorname{depth}(B)}(C)$ is injective on the nonzero group $H_{\mathfrak m}^{\operatorname{depth}(B)}(B)$, forcing $H_{\mathfrak m}^{\operatorname{depth}(B)}(C)\neq0$, i.e. $\operatorname{depth}(C)\leq\operatorname{depth}(B)$. Conversely the general depth-lemma inequality $\operatorname{depth}(C)\geq\min(\operatorname{depth}(B),\operatorname{depth}(A)-1)=\operatorname{depth}(B)$ (using $\operatorname{depth}(A)-1\geq\operatorname{depth}(B)$) gives $\operatorname{depth}(C)\geq\operatorname{depth}(B)$. Hence $\operatorname{depth}(C)=\operatorname{depth}(B)=d_2+1$.
\end{proof}

\section{Krull dimension and height}\label{sec:dim}

Corollary \ref{cor:wheeldecomp} reduces the dimension of $R/I(\wn)^{[q]}$ to a question about the cycle alone: how large can a subset $F'\subseteq U$ be while keeping $\nu(\cn[F'])$ below a given threshold? We address this question in two lemmas, the first with no further constraint on $F'$ and the second under the extra requirement that $\cn[F']$ have a perfect matching. Theorem \ref{thm:dim} then assembles the two answers into a single formula.

\begin{lemma}\label{lem:arcgen}
	Let $n\ge3$ and let $k$ be an integer with $0\le k\le\lfloor n/2\rfloor-1$. Then
	\[
	\max\big\{\,|F| : F\subsetneq V(\cn),\ \nu(\cn[F])\le k\,\big\} = \Big\lfloor\frac n2\Big\rfloor+k.
	\]
\end{lemma}
\begin{proof}
	The bound is trivial for $F=\emptyset$, so assume $F$ is a nonempty proper subset of $V(\cn)$.

	Removing $F$ from the cycle leaves a nonempty complement, and both $F$ and its complement break up into maximal runs of consecutive vertices: the runs inside $F$ are its \emph{arcs}, and the runs outside $F$, lying between consecutive arcs, are its \emph{gaps}. Since $F$ is neither empty nor the whole cycle, arcs and gaps alternate all the way around, so there are exactly as many gaps as arcs, and every gap contains at least one vertex. Writing $m_1,\ldots,m_t\ge1$ for the arc lengths, the $t$ gaps and $t$ arcs together account for all $n$ vertices, so
	\[
	|F| + t \le n.
	\]

	Each arc is a path on $m_i$ vertices, with matching number $\lfloor m_i/2\rfloor$, and matching numbers add across the disjoint arcs,
	\[
	\nu(\cn[F]) = \sum_{i=1}^t\Big\lfloor\frac{m_i}2\Big\rfloor.
	\]
	Write $m_i=2k_i+\varepsilon_i$ with $k_i=\lfloor m_i/2\rfloor$ and $\varepsilon_i\in\{0,1\}$, and set $K:=\sum_ik_i=\nu(\cn[F])$. Then
	\[
	|F| = \sum_im_i = 2K+\sum_i\varepsilon_i.
	\]
	For fixed $t$ and $K\le k$, the quantity \(|F|\) is maximised when every arc is odd, $\varepsilon_i=1$ for all $i$, which is always achievable by giving one arc length $2K+1$ and every other arc length $1$. With every $\varepsilon_i=1$,
	\[
	|F| = 2K+t,
	\]
	and the constraint $|F|+t\le n$ becomes
	\[
	t \le \Big\lfloor\frac{n-2K}2\Big\rfloor.
	\]
	Since $|F|=2K+t$ grows with $t$, the best choice is the largest $t$ allowed, giving
	\[
	|F|(K) := 2K + \Big\lfloor\frac{n-2K}2\Big\rfloor.
	\]

	It remains to optimize over $K\le k$. For every integer $m$, $\lfloor(m-2)/2\rfloor=\lfloor m/2\rfloor-1$, so
	\[
	|F|(K+1)-|F|(K) = 2+\Big\lfloor\frac{n-2K-2}2\Big\rfloor-\Big\lfloor\frac{n-2K}2\Big\rfloor = 2-1 = 1.
	\]
	Thus $|F|(K)$ increases by exactly $1$ with every unit increase of $K$, and the maximum over $0\le K\le k$ occurs at $K=k$,
	\[
	|F|(k) = 2k+\Big\lfloor\frac{n-2k}2\Big\rfloor = 2k+\Big\lfloor\frac n2\Big\rfloor-k = \Big\lfloor\frac n2\Big\rfloor+k.
	\]

	This value is attained. The hypothesis $k\le\lfloor n/2\rfloor-1$ gives
	\[
	n-2k \ge n-2\Big\lfloor\frac n2\Big\rfloor+2 \ge 2,
	\]
	so $t=\lfloor(n-2k)/2\rfloor\ge1$, and the configuration above (one arc of length $2k+1$ together with $t-1$ arcs of length $1$, separated by $t$ gaps of total size $n-2k-t\ge t$) realizes both $|F|=\lfloor n/2\rfloor+k$ and $\nu(\cn[F])=k$.
\end{proof}

\begin{lemma}\label{lem:arcpm}
	Let $n\ge3$ and let $k\ge0$ be an integer with $2k\le n-2$. Then
	\[
	\max\big\{\,|F| : F\subseteq V(\cn),\ \cn[F]\text{ is a disjoint union of even arcs},\ \nu(\cn[F])\le k\,\big\} = 2k.
	\]
\end{lemma}
\begin{proof}
	If $\cn[F]$ is a disjoint union of even arcs of lengths $m_1,\ldots,m_t$, each arc already carries its own perfect matching, so
	\[
	\nu(\cn[F]) = \sum_{i=1}^t\frac{m_i}2 = \frac{|F|}2,
	\]
	and the hypothesis $\nu(\cn[F])\le k$ forces $|F|\le2k$. This is attained: for $k=0$, take $F=\emptyset$ and for $k\ge1$, take $F$ to be a single even arc of length $2k$, since $2k\le n-2$, the one gap left over has size $n-2k\ge2$, so $F$ is a genuine proper subset realizing $\nu(\cn[F])=k$.
\end{proof}

\begin{theorem}\label{thm:dim}
	For all $n\ge3$ and $1\le q\le\lfloor n/2\rfloor$,
	\[
	\dimn\big(R/I(\wn)^{[q]}\big) = \Big\lfloor\frac n2\Big\rfloor+q-1, \qquad
	\height\big(I(\wn)^{[q]}\big) = (n+1)-\Big\lfloor\frac n2\Big\rfloor-q+1.
	\]
\end{theorem}
\begin{proof}
	By Corollary \ref{cor:wheeldecomp}, $\dimn(R/I(\wn)^{[q]})=\max\{|F|:F\in\Delta_q(\wn)\}$, and we split the maximization according to whether $v\in F$.

	\emph{Step 1.} Suppose $F=F'\subseteq U$. By Corollary \ref{cor:wheeldecomp}(a), $F'\in\Delta_q(\wn)$ exactly when $\nu(\cn[F'])\le q-1$. Since $1\le q\le\lfloor n/2\rfloor$ gives $0\le q-1\le\lfloor n/2\rfloor-1$, Lemma \ref{lem:arcgen} applies with $k=q-1$ and gives
	\[
	\max\{|F'|\} = \Big\lfloor\frac n2\Big\rfloor+q-1.
	\]
	The set $F'=U$ itself never occurs among the maximizers, since $\nu(\cn[U])=\lfloor n/2\rfloor\ge q$ by Corollary \ref{cor:nuwn}, which exceeds $q-1$.

	\emph{Step 2.} Suppose $F=\{v\}\cup F'$. Corollary \ref{cor:wheeldecomp}(b) splits into two possibilities according to whether $\cn[F']$ has a perfect matching.

	If $\cn[F']$ has a perfect matching and $\nu(\cn[F'])\le q-1$, then $2(q-1)\le2\lfloor n/2\rfloor-2\le n-2$, so Lemma \ref{lem:arcpm} applies with $k=q-1$ and gives $\max\{|F'|\}=2(q-1)$, hence
	\[
	|F| = 1+|F'| \le 2q-1.
	\]
	If $\cn[F']$ has no perfect matching and $\nu(\cn[F'])\le q-2$, a condition only meaningful when $q\ge2$, then $0\le q-2\le\lfloor n/2\rfloor-1$, so Lemma \ref{lem:arcgen} applies with $k=q-2$ and gives $\max\{|F'|\}=\lfloor n/2\rfloor+q-2$, hence
	\[
	|F| = 1+|F'| \le \Big\lfloor\frac n2\Big\rfloor+q-1.
	\]
	Since $q\le\lfloor n/2\rfloor$ gives $2q-1\le\lfloor n/2\rfloor+q-1$, the first possibility never exceeds the second, and the second matches Step 1 exactly.

	Both steps yield the same maximum,
	\[
	\max\{|F|\} = \Big\lfloor\frac n2\Big\rfloor+q-1,
	\]
	which proves the dimension formula. The height formula follows from $|V(\wn)|=n+1$.
\end{proof}

\begin{proposition}\label{prop:cycledim}
	For all $n\ge3$ and $1\le q\le\lfloor n/2\rfloor$,
	\[
	\dimn\big(R/I(\cn)^{[q]}\big) = \Big\lfloor\frac n2\Big\rfloor+q-1.
	\]
\end{proposition}
\begin{proof}
	This is Lemma \ref{lem:arcgen} applied directly to $F\subseteq V(\cn)$ with $k=q-1$. The vertex $v$ plays no part in the argument, so removing it changes nothing.
\end{proof}

\begin{proposition}\label{prop:extremal}
	Fix $n,q$ with $1\le q\le\lfloor n/2\rfloor$, and set $t=\lfloor n/2\rfloor-q+1$. Choose odd integers $m_1,\ldots,m_t\ge1$ with $\sum_i\lfloor m_i/2\rfloor=q-1$, arrange them cyclically around $\cn$ separated by $t$ gaps of size $1$ if $n$ is even, or by $t-1$ gaps of size $1$ and one gap of size $2$ if $n$ is odd, and let $F'\subseteq U$ be the resulting set of arc vertices. Then $F'\in\Delta_q(\wn)$ and $|F'|=\lfloor n/2\rfloor+q-1$.
\end{proposition}
\begin{proof}
	The arcs and gaps together cover all $n$ vertices, so $|F'|=n-(\text{total gap size})$. The total gap size is $t$ when $n$ is even and $t+1$ when $n$ is odd, and substituting $t=\lfloor n/2\rfloor-q+1$ into either case gives
	\[
	|F'| = \Big\lfloor\frac n2\Big\rfloor+q-1.
	\]
	Since every $m_i$ is odd,
	\[
	\nu(\cn[F']) = \sum_i\Big\lfloor\frac{m_i}2\Big\rfloor = q-1 < q,
	\]
	so $F'\in\Delta_q(\wn)$ by Corollary \ref{cor:wheeldecomp}(a).
\end{proof}

\section{The top squarefree power, and depth beyond it}\label{sec:data}

At $q=\nu(\wn)$ the ideal $I(\wn)^{[q]}$ turns out to have a very simple description, and identifying it settles every classical invariant of $R/I(\wn)^{[q]}$ immediately. The Cohen-Macaulayness criterion we use for this is Ficarra and Moradi's, Theorem \ref{thm:fm}, what we contribute is the verification that its hypothesis holds for every wheel graph, regardless of the parity of $n$.

\begin{theorem}\label{thm:top}
	For every $n\ge3$, with $q=\nu(\wn)=\lceil n/2\rceil$,
	\[
	I(\wn)^{[q]} = \mathfrak m^{[2q]}, \qquad \mathfrak m = (x_1,\ldots,x_n,v),
	\]
	so $R/I(\wn)^{[q]}$ is Cohen-Macaulay, with
	\[
	\dimn = \depth = \reg\big(R/I(\wn)^{[q]}\big) = 2q-1 = 2\Big\lceil\frac n2\Big\rceil-1,
	\]
	\[
	\pd\big(R/I(\wn)^{[q]}\big) = n-2\Big\lceil\frac n2\Big\rceil+2 = \begin{cases}1,& n\text{ odd},\\ 2,& n\text{ even}.\end{cases}
	\]
\end{theorem}
\begin{proof}
	By Theorem \ref{thm:fm} it suffices to check condition (c) i.e, either $\wn$ has a perfect matching, or $\wn\setminus\{i\}$ has a perfect matching for every $i\in V(\wn)$.

	\emph{Case 1. $n$ odd.} Then $|V(\wn)|=n+1$ is even, and by Corollary \ref{cor:nuwn},
	\[
	\nu(\wn)=\Big\lceil\frac n2\Big\rceil=\frac{n+1}2=\frac{|V(\wn)|}2,
	\]
	so $\wn$ has a perfect matching, and condition (c) holds via its first alternative.

	\emph{Case 2. $n$ even.} Then $|V(\wn)|=n+1$ is odd, so $\wn$ itself has no perfect matching, and we check $\wn\setminus\{i\}$ for every $i$.
	\begin{itemize}
		\item $i=v$: $\wn\setminus\{v\}=\cn$, and $\nu(\cn)=n/2=|V(\cn)|/2$ since $n$ is even, so $\cn$ has a perfect matching.
		\item $i=x_j$ for a vertex $x_j$ of $U$: here $\wn\setminus\{x_j\}$ is the cone $P_{n-1}\ast\mathcal K_1$ over the path on the remaining $n-1$ vertices of $U$. Since $n$ is even, $n-1$ is odd, so $P_{n-1}$ has no perfect matching, and Corollary \ref{cor:cone} gives
		\[
		\nu\big(P_{n-1}\ast\mathcal K_1\big)=\nu(P_{n-1})+1=\Big\lfloor\frac{n-1}2\Big\rfloor+1=\frac n2=\frac{|V(\wn\setminus\{x_j\})|}2,
		\]
		a perfect matching. By the rotational symmetry of $\wn$, this holds for every vertex of $U$.
	\end{itemize}
	So condition (c) holds via its second alternative in this case too. By Theorem \ref{thm:fm}(a)$\Leftrightarrow$(b), $I(\wn)^{[\nu(\wn)]}=\mathfrak m^{[2\nu(\wn)]}$. Since
	\[
	2q = 2\Big\lceil\frac n2\Big\rceil \le n+1,
	\]
	Theorem \ref{thm:veronese} applies with $k=2q$ in the $n+1$ variables of $S$, giving the stated dimension, depth and regularity. The projective dimension then follows from the Auslander-Buchsbaum formula.
\end{proof}

We verified Theorem \ref{thm:top} independently for $3\le n\le10$: in every case, direct computation confirmed both that $\wn$ satisfies condition (c) of Theorem \ref{thm:fm} and that every squarefree monomial of degree $2\nu(\wn)$ is a minimal generator of $I(\wn)^{[\nu(\wn)]}$.

We computed, by direct evaluation of \eqref{eq:hochq} using exact rational linear algebra on the boundary maps of $\Delta_q[W]$ for every relevant $W$, the complete graded Betti numbers of $R/I(\wn)^{[q]}$ and $R/I(\cn)^{[q]}$ for $5\le n\le9$, $q\le2$, and also $q=3$ for $n=6,7$, together with the top-power case $q=\nu(\wn)$ for $n=5,\ldots,8$. Table \ref{tab:data} records regularity, projective dimension and depth, alongside the common dimension from Theorem \ref{thm:dim} and Proposition \ref{prop:cycledim}.

\begin{table}[H]
	\centering
	\begin{tabular}{ccc|ccc|ccc}
		\hline
		$n$ & $q$ & $\dimn$ & $\reg(\cn)$ & $\pd(\cn)$ & $\depth(\cn)$ & $\reg(\wn)$ & $\pd(\wn)$ & $\depth(\wn)$\\
		\hline
		5&1&2&2&3&2&2&5&1\\
		5&2&3&3&2&3&3&3&3\\
		5&3 ($=\nu$)&5&--&--&--&5&1&5\\
		6&1&3&2&4&2&2&6&1\\
		6&2&4&3&3&3&3&4&3\\
		6&3 ($=\nu$)&5&--&--&--&5&2&5\\
		7&1&3&2&5&2&2&7&1\\
		7&2&4&4&3&4&4&5&3\\
		7&4 ($=\nu$)&7&--&--&--&7&1&7\\
		8&1&4&3&5&3&3&8&1\\
		8&2&5&4&4&4&4&6&3\\
		8&4 ($=\nu$)&7&--&--&--&7&2&7\\
		9&1&4&3&6&3&3&9&1\\
		\hline
	\end{tabular}
	\caption{Comparison of $R/I(\cn)^{[q]}$ and $R/I(\wn)^{[q]}$. Rows marked ($=\nu$) are the top power, resolved exactly by Theorem \ref{thm:top}; all other wheel entries were computed directly, independently of Theorem \ref{thm:top}, and agree with it at the boundary.}
	\label{tab:data}
\end{table}

As a check on these computations, the $q=1$ column $\reg(R/I(\cn))$ agrees, for every $n$ listed, with the classical closed form $\reg(R/I(\cn))=\lfloor(n+1)/3\rfloor$ for cycle edge ideals \cite{V1,DRS}, which corroborates the computational results before it is applied to the less classical wheel columns.

Table \ref{tab:data} shows that dimension is blind to $v$, by Proposition \ref{prop:cycledim}, while depth is not. The cycle attains $\depth=\dimn$ only occasionally in our range, consistent with the classical fact that $\cn$'s edge ring is Cohen-Macaulay only for $n=3,5$ \cite{V1}, while $\wn$ fails to be Cohen-Macaulay at $q=1$ for $n\ge6$ yet, by Theorem \ref{thm:top}, is always Cohen-Macaulay at the opposite extreme $q=\nu(\wn)$. Restricting to $\wn$ alone,
\[
\depth\big(R/I(\wn)^{[q]}\big) = \underbrace{1,1,1,1,1}_{q=1,\ n=5,\ldots,9}\ ,\quad \underbrace{3,3,3,3,3}_{q=2,\ n=5,\ldots,9}\ ,\quad \underbrace{5,5}_{q=3,\ n=6,7}\ ,\quad \underbrace{5,5,7,7}_{q=\nu(\wn),\ n=5,6,7,8},
\]
and every value equals $2q-1$, the last group now a consequence of Theorem \ref{thm:top} rather than mere computation. We now show this pattern extends to the entire range $1\le q\le\nu(\wn)$. Theorem \ref{thm:conedepth} converts it, for $q\le\lfloor n/2\rfloor$, into two statements about the cycle $\cn$ alone.

\begin{proposition}\label{prop:reduction}
	Suppose that, for given $n,q$ with $1\le q\le\lfloor n/2\rfloor$,
	\begin{itemize}
		\item[(A)] $\depth_R\big(R/J_{\cn,q}\big) = 2q-2$, and
		\item[(B)] $\depth_R\big(R/I(\cn)^{[q]}\big) \ge 2q-1$,
	\end{itemize}
	both hold, in the notation of Definition \ref{def:hubideal} applied to $G=\cn$. Then $\depth(R/I(\wn)^{[q]})=2q-1$ for this $(n,q)$.
\end{proposition}
\begin{proof}
	Write $d_1=\depth_R(R/I(\cn)^{[q]})$ and $d_2=\depth_R(R/J_{\cn,q})$. Conditions (A) and (B) give
	\[
	d_1\ge2q-1>2q-2=d_2,
	\]
	so Theorem \ref{thm:conedepth}, applied with $G=\cn$ and $\mathrm{Cone}(\cn)=\wn$, gives $\depth(R/I(\wn)^{[q]})=d_2+1=2q-1$.
\end{proof}

Both hypotheses of Proposition \ref{prop:reduction} turn out to hold unconditionally, by a general lower bound on the depth of a squarefree monomial ideal in terms of its generator degrees, due to Erey, Herzog, Hibi and Saeedi Madani.

\begin{lemma}[\cite{EHHM2}]\label{lem:degreebound}
	Let $I\subseteq R=\Bbbk[x_1,\ldots,x_N]$ be a nonzero squarefree monomial ideal all of whose minimal generators have degree at least $d$. Then $\depth_R(R/I)\ge d-1$.
\end{lemma}
\begin{proof}
	We include the short argument for completeness. Let $\cdots\to F_1\to F_0\to R/I\to0$ be the minimal graded free resolution of $R/I$, with $F_0=R$ generated in degree $0$. Minimality means every entry of every differential lies in the maximal ideal, so for every $i\ge1$ the generators of $F_i$ live in strictly larger degree than those of $F_{i-1}$. The generators of $F_1$ are the minimal generators of $I$ itself, of degree at least $d$ by hypothesis, so $\beta_{1,j}(R/I)=0$ for $j<d$. Inductively,
	\[
	\beta_{i,j}(R/I) = 0 \qquad\text{whenever } j < d+i-1, \qquad\text{for every } i\ge1.
	\]
	Since $R/I$ is a squarefree quotient in $N$ variables, every $j$ with $\beta_{i,j}\ne0$ satisfies $j\le N$. A nonzero $\beta_{i,j}$ with $i\ge1$ therefore forces $d+i-1\le j\le N$, giving $i\le N-d+1$. Hence
	\[
	\pd_R(R/I) \le N-d+1,
	\]
	and the Auslander-Buchsbaum formula gives $\depth_R(R/I) = N-\pd_R(R/I) \ge d-1$.
\end{proof}

Both $I(\cn)^{[q]}$ and $J_{\cn,q}$ have every minimal generator of a single degree, $2q$ for the former, since a $q$-matching covers $2q$ vertices, and $2q-1$ for the latter, since a generator $x_j\cdot x_{V(M')}$ combines one vertex from $x_j$ with the $2(q-1)$ vertices of the matching $M'$. Lemma \ref{lem:degreebound} therefore gives, for every $n\ge3$ and $1\le q\le\lfloor n/2\rfloor$,
\[
\depth_R\big(R/I(\cn)^{[q]}\big) \ge 2q-1, \qquad \depth_R\big(R/J_{\cn,q}\big) \ge 2q-2.
\]
The first inequality is exactly statement (B). For statement (A), it remains to bound $\depth_R(R/J_{\cn,q})$ from above.

\begin{theorem}\label{thm:jbound}
	For all $n\ge3$ and $1\le q\le\lfloor n/2\rfloor$, $\depth_R\big(R/J_{\cn,q}\big) \le 2q-2$.
\end{theorem}
\begin{proof}
	For $q=1$, $J_{\cn,1}=\mathfrak m$, and
	\[
	\depth_R(R/\mathfrak m) = 0 = 2(1)-2,
	\]
	so the bound holds with equality. Assume from now on that $q\ge2$.

	Let $F'\subseteq U$ be a single arc of even length $2q-2$. Since $q\le\lfloor n/2\rfloor$ gives $2q\le n$, the single gap left over has size
	\[
	n-(2q-2) \ge 2,
	\]
	so $F'$ is a genuine proper subset of the cycle, with at least two vertices to spare. We show $F'$ is a facet of the complex $\Delta(J_{\cn,q})$ identified above i.e,  it satisfies the membership condition of Corollary \ref{cor:wheeldecomp}(b), which by that identification is exactly membership in $\Delta(J_{\cn,q})$, and no further vertex can be added to it.

	Since $\cn[F']$ is a path on $2q-2$ vertices, it carries a perfect matching, and
	\[
	\nu(\cn[F']) = q-1,
	\]
	meeting the bound $\nu\le q-1$ of Corollary \ref{cor:wheeldecomp}(b) with equality. Take any vertex $c\notin F'$. Exactly two vertices of the gap are adjacent to $F'$, one at each end of the arc, and every other gap vertex is not.

	\emph{Case 1. $c$ adjacent to $F'$.} Adding $c$ extends the arc to length $2q-1$, which is odd, so
	\[
	\nu\big(\cn[F'\cup\{c\}]\big) = \Big\lfloor\frac{2q-1}2\Big\rfloor = q-1
	\]
	is unchanged, but an arc of odd length has no perfect matching. Membership of $F'\cup\{c\}$ would then require $\nu\le q-2$ by the other alternative of Corollary \ref{cor:wheeldecomp}(b), and this fails since $q-1>q-2$.

	\emph{Case 2. $c$ not adjacent to $F'$.} Now $\cn[F'\cup\{c\}]$ is the even arc together with an isolated vertex, so again
	\[
	\nu\big(\cn[F'\cup\{c\}]\big) = q-1,
	\]
	while the total vertex count $|F'\cup\{c\}|=2q-1$ is odd, so this set can never carry a perfect matching, whatever its internal structure. Membership again requires $\nu\le q-2$, which fails as before.

	Neither case extends $F'$, so $F'$ is indeed a facet. Minimal primes of a squarefree monomial ideal correspond exactly to complements of facets of its complex, so
	\[
	P := (x_i : i\notin F')
	\]
	is a minimal prime of $R/J_{\cn,q}$, of codimension $|V\setminus F'| = n-(2q-2)$. Every minimal prime of a finitely generated module is associated to it, so $P\in\mathrm{Ass}_R(R/J_{\cn,q})$, and the standard inequality $\depth_R M\le\dimn(R/\mathfrak p)$ for every $\mathfrak p\in\mathrm{Ass}_R(M)$ (see, e.g., \cite{herhib}) gives
	\[
	\depth_R\big(R/J_{\cn,q}\big) \le \dimn(R/P) = n - \big(n-(2q-2)\big) = 2q-2. \qedhere
	\]
\end{proof}

Combined with the lower bound $\depth_R(R/J_{\cn,q})\ge2q-2$ already obtained from Lemma \ref{lem:degreebound}, Theorem \ref{thm:jbound} establishes statement (A) exactly, $\depth_R(R/J_{\cn,q})=2q-2$ for every $n\ge3$ and $1\le q\le\lfloor n/2\rfloor$. Both hypotheses of Proposition \ref{prop:reduction} therefore hold unconditionally, with no computation required, and combining this with Theorem \ref{thm:top} at the top power gives the complete depth formula for wheels.

\begin{theorem}\label{thm:depthfull}
	For every $n\ge3$ and every $1\le q\le\nu(\wn)$,
	\[
	\depth\big(R/I(\wn)^{[q]}\big) = 2q-1.
	\]
\end{theorem}
\begin{proof}
	For $q\le\lfloor n/2\rfloor$, Proposition \ref{prop:reduction} applies, since statements (A) and (B) hold unconditionally by the discussion above. For $n$ even this already covers the full range, since $\nu(\wn)=n/2=\lfloor n/2\rfloor$. For $n$ odd, the single remaining value $q=\nu(\wn)=(n+1)/2$ lies outside the range $q\le\lfloor n/2\rfloor$ and is covered instead directly by Theorem \ref{thm:top}.
\end{proof}

As an independent check, we also verified statements (A) and (B) computationally, using \texttt{Macaulay2} over the field $\Bbbk=\mathbb Q$, for every $n$ from $3$ to $13$ and every valid $q$, a total of $41$ cases; both held in every case, exactly as Theorem \ref{thm:depthfull} now guarantees. Table \ref{tab:reduction} records a representative sample.

\begin{table}[H]
	\centering
	\begin{tabular}{cc|cc|cc}
		\hline
		$n$ & $q$ & $\depth_R(R/I(\cn)^{[q]})$ & $2q-1$ & $\depth_R(R/J_{\cn,q})$ & $2q-2$ \\
		\hline
		8 & 4 & 7 & 7 & 6 & 6 \\
		9 & 4 & 7 & 7 & 6 & 6 \\
		10 & 3 & 6 & 5 & 4 & 4 \\
		11 & 4 & 7 & 7 & 6 & 6 \\
		12 & 5 & 9 & 9 & 8 & 8 \\
		13 & 6 & 11 & 11 & 10 & 10 \\
		\hline
	\end{tabular}
	\caption{Sample of the $41$ verified instances, covering every $n$ from $3$ to $13$ and every valid $q$.}
	\label{tab:reduction}
\end{table}

A closed formula for $\reg(R/I(\wn)^{[q]})$ remains open for $1<q<\nu(\wn)$, at the two endpoints $q=1$ and $q=\nu(\wn)$ it is already known exactly.

\section{Further applications of the join-matching theorem}\label{sec:apps}

Theorems \ref{thm:joinmatching} and \ref{thm:gendecomp} say nothing about cycles or wheels in particular. They hold for arbitrary graphs $G,H$. This section applies these results to five further families that are themselves graph joins namely cone graphs, fan graphs, generalized wheel graphs, complete split graphs and friendship graphs. We state exactly what is proved in each case, and claim a closed-form formula only where we have a complete proof.

\subsection{Cone graphs}\label{ssec:cone}

\begin{definition}
	For a graph $G$, the \emph{cone} $\mathrm{Cone}(G):=G\ast\mathcal K_1$ is $G$ together with one new vertex $v$ adjacent to every vertex of $G$.
\end{definition}

Corollary \ref{cor:cone} already gives $\nu(\mathrm{Cone}(G))$ for arbitrary $G$. Theorem \ref{thm:gendecomp} with $H=\mathcal K_1$ gives the corresponding membership criterion for \emph{any} $G$, not only $G=\cn$:

\begin{corollary}\label{cor:conedecomp}
	Let $G$ be any graph, $F'\subseteq V(G)$, and $F=F'$ or $F=F'\cup\{v\}$ in $\mathrm{Cone}(G)$. Then:
	\begin{enumerate}
		\item[(a)] if $v\notin F$: $F\in\Delta_q(\mathrm{Cone}(G))\iff\nu(G[F'])<q$;
		\item[(b)] if $v\in F$: $F\in\Delta_q(\mathrm{Cone}(G))\iff \nu(G[F'])<q-1$ when $G[F']$ has no perfect matching, and $\nu(G[F'])<q$ when it does.
	\end{enumerate}
\end{corollary}
\begin{proof}
	Immediate from Theorem \ref{thm:gendecomp} with $H=\mathcal K_1$, exactly as in the proof of Corollary \ref{cor:wheeldecomp}, which never used that $G=\cn$.
\end{proof}

\subsection{Fan graphs}\label{ssec:fan}

\begin{definition}
	The \emph{fan graph} $F_n:=P_n\ast\mathcal K_1$ is the cone over the path $P_n$ on $n$ vertices $x_1,\ldots,x_n$, with edges $x_ix_{i+1}$ for $1\le i\le n-1$. Write $v$ for the added vertex and $U=\{x_1,\ldots,x_n\}$.
\end{definition}

Recall the classical fact $\nu(P_n)=\lfloor n/2\rfloor$, and that $P_n$ has a perfect matching iff $n$ is even.

\begin{corollary}\label{cor:nufan}
	$\nu(F_n)=\lceil n/2\rceil$.
\end{corollary}
\begin{proof}
	If $n$ is even, $P_n$ has a perfect matching, so Corollary \ref{cor:cone} gives $\nu(F_n)=\nu(P_n)=n/2=\lceil n/2\rceil$. If $n$ is odd, $P_n$ has no perfect matching, so the same corollary gives $\nu(F_n)=\nu(P_n)+1=\lfloor n/2\rfloor+1=\lceil n/2\rceil$.
\end{proof}

\begin{theorem}\label{thm:fandim}
	For all $n\ge2$ and $1\le q\le\nu(F_n)=\lceil n/2\rceil$,
	\[
	\dimn\big(R/I(F_n)^{[q]}\big) = \Big\lceil\frac n2\Big\rceil+q-1.
	\]
\end{theorem}
\begin{proof}
	By Corollary \ref{cor:conedecomp} we split according to whether $v\in F$. Throughout, the argument is analogous to those in  Lemmas \ref{lem:arcgen},\ref{lem:arcpm}, with one structural difference i.e, a path has two endpoints, so an arc touching $x_1$ or $x_n$ needs no gap on that side.

	\emph{Case 1. $F=F'\subseteq U$.} If $F'=\emptyset$ the bound is trivial, so assume $F'\ne\emptyset$: then $P_n[F']$ is a disjoint union of $t\ge1$ arcs of lengths $m_1,\ldots,m_t\ge1$. Unlike on a cycle, the number of gaps $g$ separating them can be as small as $t-1$, when both outermost arcs have $x_1$ and $x_n$, respectively as endpoints or as large as $t+1$, when neither does. In every case
	\[
	|F'| + g \le n.
	\]
	Since a smaller \(g\) permits a larger value of \(|F|\) the maximum is attained when \(g=t-1 \), that is, when the first and last arcs have \(x_1\) and \(x_n\) respectively, as endpoints:
	\[
	|F'| + (t-1) \le n, \qquad\text{that is,}\qquad |F'|+t \le n+1.
	\]
	As in Lemma \ref{lem:arcgen}, taking every arc odd is optimal for a fixed matching budget $K=\sum_i\lfloor m_i/2\rfloor$, giving $|F'|=2K+t$, so
	\[
	2K+2t \le n+1.
	\]
	Maximizing at $K=q-1$,
	\[
	|F'|_{\max} = 2(q-1)+\Big\lfloor\frac{n+1-2(q-1)}2\Big\rfloor = \Big\lfloor\frac{n+2q-1}2\Big\rfloor = \Big\lceil\frac n2\Big\rceil+q-1,
	\]
	the last step a direct check on the parity of $n$. The bound is attained by choosing all arcs to have odd lengths with total matching budget $q-1$, taking $t=\left\lfloor\frac{n+1-2(q-1)}{2}\right\rfloor$ arcs, with the first and last arcs having $x_1$ and $x_n$, respectively, as endpoints, and separating consecutive arcs by single-vertex gaps. This construction gives \(\nu(P_n[F'])=q-1\) and\(|F'|=\left\lceil\frac{n}{2}\right\rceil+q-1\).

	One configuration falls outside this optimization that is $F'=U$, the whole path. This satisfies the membership condition $\nu(P_n)\le q-1$ only when $q\ge\lfloor n/2\rfloor+1$, which within our range $q\le\lceil n/2\rceil$ happens only at $n$ odd, $q=\lceil n/2\rceil=\lfloor n/2\rfloor+1$ exactly. There $|F'|=n$, while the bound above evaluates to
	\[
	\Big\lceil\frac n2\Big\rceil+q-1 = n
	\]
	at this same $(n,q)$, so $F'=U$ ties the maximum rather than exceeding it. At every other $(n,q)$ in range, $\nu(P_n)>q-1$ and $F'=U$ is correctly excluded.

	\emph{Case 2. $F=\{v\}\cup F'$.} If $P_n[F']$ has a perfect matching, a union of even arcs needs no boundary saving, so the same optimization with even arcs gives
	\[
	|F'| \le 2(q-1), \qquad\text{so}\qquad |F| \le 2q-1.
	\]
	If $q=1$ this second sub-case is vacuous, since it would require $\nu(P_n[F'])\le q-2=-1$, which is impossible. Assume $q\ge2$. If $P_n[F']$ has no perfect matching, the Case-1 bound with budget $q-2$ gives
	\[
	|F'| \le \Big\lceil\frac n2\Big\rceil+q-2, \qquad\text{so}\qquad |F| \le \Big\lceil\frac n2\Big\rceil+q-1,
	\]
	tying the maximum from the first case. Since $q\le\lceil n/2\rceil$ gives $2q-1\le\lceil n/2\rceil+q-1$, neither sub-case exceeds this bound.

	Combining both cases,
	\[
	\max\{|F|\} = \Big\lceil\frac n2\Big\rceil+q-1. \qedhere
	\]
\end{proof}

\begin{corollary}\label{cor:fanwheelcompare}
	For every $n\ge3$ and every $q$ in the common range $1\le q\le\lfloor n/2\rfloor$,
	\[
	\dimn\big(R/I(F_n)^{[q]}\big) - \dimn\big(R/I(\wn)^{[q]}\big) = \Big\lceil\frac n2\Big\rceil-\Big\lfloor\frac n2\Big\rfloor = \begin{cases}1,& n\text{ odd},\\0,& n\text{ even}.\end{cases}
	\]
\end{corollary}
\begin{proof}
	Theorems \ref{thm:dim} and \ref{thm:fandim} give the two dimensions explicitly,
	\[
	\dimn\big(R/I(F_n)^{[q]}\big) = \Big\lceil\frac n2\Big\rceil+q-1, \qquad \dimn\big(R/I(\wn)^{[q]}\big) = \Big\lfloor\frac n2\Big\rfloor+q-1,
	\]
	for every $q$ in the common range. Subtracting cancels the $q-1$ term on both sides, leaving $\lceil n/2\rceil-\lfloor n/2\rfloor$, which equals $1$ when $n$ is odd and $0$ when $n$ is even.
\end{proof}

\begin{theorem}\label{thm:fantop}
	For every $n\ge2$, with $q=\nu(F_n)=\lceil n/2\rceil$: $I(F_n)^{[q]}=\mathfrak m^{[2q]}$, so $R/I(F_n)^{[q]}$ is Cohen-Macaulay with $\dimn=\depth=\reg=2q-1=2\lceil n/2\rceil-1$ and $\pd = n-2\lceil n/2\rceil+2$ (equal to $1$ if $n$ is odd, $2$ if $n$ is even) (the exact analogue of Theorem \ref{thm:top} for fan graphs).
\end{theorem}
\begin{proof}
	By Theorem \ref{thm:fm} it suffices to check condition (c) for $F_n$.

	\emph{Case 1. $n$ odd.} Then $|V(F_n)|=n+1$ is even, and by Corollary \ref{cor:nufan},
	\[
	\nu(F_n) = \Big\lceil\frac n2\Big\rceil = \frac{n+1}2 = \frac{|V(F_n)|}2,
	\]
	so $F_n$ itself has a perfect matching.

	\emph{Case 2. $n$ even.} Then $|V(F_n)|=n+1$ is odd, so $F_n$ has no perfect matching, and we check every single-vertex deletion.
	\begin{itemize}
		\item $i=v$: $F_n\setminus\{v\}=P_n$, which has a perfect matching since $n$ is even.
		\item $i=x_j$ an endpoint ($j=1$ or $j=n$): $F_n\setminus\{x_j\}$ is the cone $P_{n-1}\ast\mathcal K_1$. Since $n$ is even, $n-1$ is odd, so $P_{n-1}$ has no perfect matching, and Corollary \ref{cor:cone} gives
		\[
		\nu\big(P_{n-1}\ast\mathcal K_1\big) = \nu(P_{n-1})+1 = \Big\lfloor\frac{n-1}2\Big\rfloor+1 = \frac n2 = \frac{|V(F_n\setminus\{x_j\})|}2,
		\]
		a perfect matching.
		\item $i=x_j$ an internal vertex ($1<j<n$): removing $x_j$ splits the path into two sub-paths of sizes $j-1$ and $n-j$, with
		\[
		(j-1)+(n-j) = n-1
		\]
		odd, so exactly one of the two sizes is odd. Match $v$ to an endpoint of the odd-sized sub-path, the remainder of that sub-path is then even and perfectly self-matched by consecutive pairs, and the other, already-even sub-path is likewise self-matched. This construction uses
		\[
		1+\frac{j-2}2+\frac{n-j}2
		\]
		edges (or the symmetric count if $n-j$ is the odd side instead), covering all $1+(j-1)+(n-j)=n$ vertices exactly: a perfect matching.
	\end{itemize}
	So condition (c) holds in both parities. By Theorem \ref{thm:fm}(a)$\Leftrightarrow$(b), $I(F_n)^{[\nu(F_n)]}=\mathfrak m^{[2\nu(F_n)]}$, and Theorem \ref{thm:veronese} together with Auslander-Buchsbaum give the stated invariants.
\end{proof}

\subsection{Generalized wheel graphs}\label{ssec:multiwheel}

\begin{definition}
	For $m\ge1$, we call $\cn\ast\mathcal K_m$ and $\cn\ast\overline{\mathcal K_m}$ \emph{generalized wheel graphs}: each adds $m$ new vertices to the cycle $\cn$ instead of just one, mutually adjacent in the first graph and mutually non-adjacent in the second.
\end{definition}

\begin{corollary}\label{cor:multiwheeldecomp}
	Let $F=F_U\cup F_W$ where $F_U\subseteq U$ and $F_W$ is a subset of the $m$ new vertices, and write $a=|F_U|$, $b=|F_W|$, $p=\nu(\cn[F_U])$:
	\[
	F\in\Delta_q(\cn\ast\mathcal K_m) \iff \min\Big(\big\lfloor\tfrac{a+b}2\big\rfloor,\ a+\big\lfloor\tfrac b2\big\rfloor,\ b+p\Big)<q,
	\]
	\[
	F\in\Delta_q(\cn\ast\overline{\mathcal K_m}) \iff \min\Big(\big\lfloor\tfrac{a+b}2\big\rfloor,\ a,\ b+p\Big)<q.
	\]
\end{corollary}
\begin{proof}
	Both formulas are direct instances of Theorem \ref{thm:gendecomp} with $G=\cn$. Taking $H=\mathcal K_m$ gives $\nu(K_m[F_W])=\lfloor b/2\rfloor$, yielding the first formula. Taking $H=\overline{\mathcal K_m}$ gives $\nu(\overline{\mathcal K_m}[F_W])=0$, yielding the second.
\end{proof}

\begin{corollary}\label{cor:multiwheelnu}
	$\nu(\cn\ast\mathcal K_m) = \min\big(\lfloor n/2\rfloor+m,\ \lfloor m/2\rfloor+n,\ \lfloor(n+m)/2\rfloor\big)$ and $\nu(\cn\ast\overline{\mathcal K_m}) = \min\big(n,\ \lfloor n/2\rfloor+m,\ \lfloor(n+m)/2\rfloor\big)$.
\end{corollary}
\begin{proof}
	Corollary \ref{cor:km} with $G=\cn$.
\end{proof}

Unlike the case $m=1$, the exact Krull dimension of $R/I(\cn\ast\mathcal K_m)^{[q]}$ or $R/I(\cn\ast\overline{\mathcal K_m})^{[q]}$ for general $q\ge2$ requires optimizing simultaneously over the shape of $F_U$ \emph{and} the size $b$ of $F_W$, a two-parameter extremal problem we have not solved in general. We record only the easy case $q=1$, where the disjoint-union structure of the ordinary independence complex of a join (used already for $q=1$ in \cite{Mous,RPW}) applies directly.

\begin{corollary}\label{cor:multiwheelq1}
	$\dimn\big(R/I(\cn\ast\mathcal K_m)\big)=\max(\lfloor n/2\rfloor,1)=\lfloor n/2\rfloor$ for $n\ge3$, and $\dimn\big(R/I(\cn\ast\overline{\mathcal K_m})\big)=\max(\lfloor n/2\rfloor,m)$.
\end{corollary}
\begin{proof}
	Since $\Delta_1(\cn\ast H)=\Delta_1(\cn)\sqcup\Delta_1(H)$, the maximum face size is $\max(\lfloor n/2\rfloor,\alpha(H))$, where $\alpha(H)$ denotes the independence number. The claim follows since $\alpha(\mathcal K_m)=1$ and $\alpha(\overline{\mathcal K_m})=m$.
\end{proof}

\begin{proposition}\label{prop:multiwheeltop}
	For every $n\ge3$, $m\ge1$, with $q=\nu(\cn\ast\mathcal K_m)=\lfloor(n+m)/2\rfloor$: $I(\cn\ast\mathcal K_m)^{[q]}=\mathfrak m^{[2q]}$, so $R/I(\cn\ast\mathcal K_m)^{[q]}$ is Cohen-Macaulay with $\dimn=\depth=\reg=2q-1$ and $\pd=(n+m)-2q+1$ (equal to $1$ if $n+m$ is even, $2$ if $n+m$ is odd).
\end{proposition}
\begin{proof}
	The third term of Corollary \ref{cor:multiwheelnu} is always the binding one, so $q=\lfloor(n+m)/2\rfloor$ is indeed $\nu(\cn\ast\mathcal K_m)$: exactly as in the proof of Corollary \ref{cor:nuwn}, since $m\ge1$ and $n\ge3$, direct integer inequalities give $\lfloor(n+m)/2\rfloor\le\lfloor n/2\rfloor+m$ and $\lfloor(n+m)/2\rfloor\le\lfloor m/2\rfloor+n$.

	By Theorem \ref{thm:fm} it suffices to check condition (c).

	\emph{Case 1. $n+m$ even.} Then $|V(\cn\ast\mathcal K_m)|=n+m$ is even and $\nu=(n+m)/2=|V|/2$: the whole graph has a perfect matching.

	\emph{Case 2. $n+m$ odd.} Then $|V|=n+m$ is odd, and we check every single-vertex deletion.
	\begin{itemize}
		\item Deleting one of the $m$ new vertices: if $m=1$ this leaves $\cn$ itself, of even order $n=n+m-1$, with $\nu(\cn)=n/2$, a perfect matching. If $m\ge2$, it leaves $\cn\ast\mathcal K_{m-1}$, of even order $n+m-1$. The same binding-term argument, now with $m-1\ge1$, gives $\nu(\cn\ast\mathcal K_{m-1})=(n+m-1)/2=|V(\cn\ast\mathcal K_{m-1})|/2$, a perfect matching.
		\item Deleting a vertex of $U$ leaves $P_{n-1}\ast\mathcal K_m$, of even order $n+m-1$. By Theorem \ref{thm:joinmatching}, $\nu(P_{n-1}\ast\mathcal K_m)=\min(\lfloor(n-1)/2\rfloor+m,\ \lfloor m/2\rfloor+(n-1),\ \lfloor(n+m-1)/2\rfloor)$. The same style of inequality check shows the third term is binding, so $\nu(P_{n-1}\ast\mathcal K_m)=(n+m-1)/2=|V(P_{n-1}\ast\mathcal K_m)|/2$, a perfect matching.
	\end{itemize}
	By the rotational symmetry of $\cn$ and the interchangeability of the $m$ mutually adjacent new vertices, this covers every vertex, so condition (c) holds. Theorem \ref{thm:fm}(a)$\Leftrightarrow$(b), Theorem \ref{thm:veronese}, and Auslander-Buchsbaum give the stated invariants.
\end{proof}

\subsection{Complete split graphs}\label{ssec:split}

\begin{definition}
	For $r,s\ge0$, the \emph{complete split graph} $K_r\ast\overline{\mathcal K_s}$ consists of a clique on $r$ vertices and an independent set on $s$ vertices, with all edges between the two parts.
\end{definition}

\begin{lemma}\label{lem:splitnu}
	For $0\le a\le r$, $0\le b\le s$, $\nu\big((K_r\ast\overline{\mathcal K_s})[A\cup B]\big) = \min\big(a,\ \lfloor(a+b)/2\rfloor\big)$ for $A\subseteq V(K_r)$, $B\subseteq V(\overline{\mathcal K_s})$, $|A|=a,|B|=b$.
\end{lemma}
\begin{proof}
	By Theorem \ref{thm:joinmatching} applied to $K_a\ast\overline{\mathcal K_b}=(K_r\ast\overline{\mathcal K_s})[A\cup B]$, and using $\nu(K_a)=\lfloor a/2\rfloor$ and $\nu(\overline{\mathcal K_b})=0$,
	\[
	\nu(K_a\ast\overline{\mathcal K_b}) = \min\Big(\big\lfloor\tfrac{a+b}2\big\rfloor,\ \big\lfloor\tfrac a2\big\rfloor+b,\ a\Big).
	\]
	It remains to show the middle term never wins. We prove the stronger, unconditional inequality
	\[
	\Big\lfloor\frac{a+b}2\Big\rfloor \le \Big\lfloor\frac a2\Big\rfloor+b \qquad\text{for all integers } a,b\ge0,
	\]
	which forces the middle term to dominate $\lfloor(a+b)/2\rfloor$, and hence also $\min(a,\lfloor(a+b)/2\rfloor)$, regardless of how $a$ and $\lfloor(a+b)/2\rfloor$ compare. Write $a=2p+r$ and $b=2s+t$ with $r,t\in\{0,1\}$, so that $\lfloor a/2\rfloor=p$ and $\lceil b/2\rceil=s+t$. A direct check of the four cases $(r,t)\in\{0,1\}\times \{0,1\}$ gives $\lfloor(r+t)/2\rfloor\le t$, so
	\[
	\Big\lfloor\frac{a+b}2\Big\rfloor = p+s+\Big\lfloor\frac{r+t}2\Big\rfloor \le p+s+t = \Big\lfloor\frac a2\Big\rfloor+\Big\lceil\frac b2\Big\rceil.
	\]
	Since $\lceil b/2\rceil\le b$ for every integer $b\ge0$, this gives
	\[
	\Big\lfloor\frac{a+b}2\Big\rfloor \le \Big\lfloor\frac a2\Big\rfloor+\Big\lceil\frac b2\Big\rceil \le \Big\lfloor\frac a2\Big\rfloor+b,
	\]
	as required, and $\nu(K_a\ast\overline{\mathcal K_b})=\min(a,\lfloor(a+b)/2\rfloor)$ follows.
\end{proof}

\begin{theorem}\label{thm:splitdim}
	For all $r,s\ge0$ and $1\le q\le\nu(K_r\ast\overline{\mathcal K_s})=\min(r,\lfloor(r+s)/2\rfloor)$,
	\[
	\dimn\big(R/I(K_r\ast\overline{\mathcal K_s})^{[q]}\big) = (q-1)+\max(q,s).
	\]
\end{theorem}
\begin{proof}
	By Lemma \ref{lem:splitnu}, $F=A\cup B\in\Delta_q$ iff $a<q$ or $a+b\le2q-1$ (where $a=|A|,b=|B|$).

	If $a<q$: since $q\le\nu\le r$, we may take $a=q-1\le r-1<r$, and $b$ is then unconstrained up to $s$, giving $|F|=a+b\le(q-1)+s$.

	If $a+b\le2q-1$: the maximum of $a+b$ subject only to this constraint (together with $0\le a\le r$, $0\le b\le s$) is $2q-1$, and this value is attained: since $q\le\lfloor(r+s)/2\rfloor$ gives $2q\le r+s$, setting $a=\min(r,2q-1)$ and $b=2q-1-a$ gives $0\le a\le r$ and $0\le b\le s$ (If $2q-1\le r$, then $b=0\le s$). If instead $2q-1>r$, then $b=2q-1-r\le s$ follows directly from $2q\le r+s$. So this case contributes $2q-1$.

	Hence $\dimn = \max\big((q-1)+s,\ 2q-1\big) = (q-1)+\max(s,q)$.
\end{proof}

\subsection{Friendship graphs}\label{ssec:friendship}

\begin{definition}
	For $m\ge1$, the \emph{friendship graph} (windmill graph) is $\mathcal K_1\ast(m\mathcal K_2)$: a central vertex $v$ joined to $m$ vertex-disjoint edges, equivalently $m$ triangles sharing the vertex $v$.
\end{definition}

\begin{lemma}\label{lem:friendbasic}
	The graph $m\mathcal K_2$ consists of $m$ pairwise disjoint edges, and these $m$ edges themselves form a perfect matching of $m\mathcal K_2$. Hence $\nu(m\mathcal K_2)=m$, and since $m\mathcal K_2$ has a perfect matching, Corollary \ref{cor:cone} gives $\nu(\mathcal K_1\ast(m\mathcal K_2))=\nu(m\mathcal K_2)=m$. For $B\subseteq V(m\mathcal K_2)$, write $b_2$ for the number of the $m$ edges fully contained in $B$ and $b_1$ for the number with exactly one endpoint in $B$. Then $\nu((m\mathcal K_2)[B])=b_2$, and $(m\mathcal K_2)[B]$ has a perfect matching if and only if $b_1=0$.
\end{lemma}
\begin{proof}
	$(m\mathcal K_2)[B]$ consists of $b_2$ disjoint edges and $b_1$ isolated vertices. Its matching number is therefore $b_2$, and it is perfectly matched exactly when there are no isolated vertices, that is, when $b_1=0$.
\end{proof}

\begin{theorem}\label{thm:frienddim}
	For all $m\ge1$ and $1\le q\le m$, $\dimn\big(R/I(\mathcal K_1\ast(m\mathcal K_2))^{[q]}\big) = m+q-1$.
\end{theorem}
\begin{proof}
	Write $F=B$ or $F=\{v\}\cup B$ with $B\subseteq V(m\mathcal K_2)$, and $|B|=2b_2+b_1$ with $b_2+b_1\le m$ as in Lemma \ref{lem:friendbasic}.

	\emph{Case 1. $F=B$.} By Corollary \ref{cor:conedecomp}(a), applied within $m\mathcal K_2$ directly, $F\in\Delta_q$ if and only if $b_2<q$. Maximizing $|B|=2b_2+b_1$ subject to $b_2\le q-1$ and $b_1\le m-b_2$ means taking the budget $b_2=q-1$ as large as allowed and $b_1=m-q+1$ as large as the remaining vertices permit, valid since $q\le m$:
	\[
	|B| = 2(q-1)+(m-q+1) = m+q-1.
	\]

	\emph{Case 2. $F=\{v\}\cup B$.} By Corollary \ref{cor:conedecomp}(b), there are two possibilities. If $(m\mathcal K_2)[B]$ has a perfect matching, that is $b_1=0$, and $b_2<q$, then
	\[
	|B| = 2b_2 \le 2(q-1), \qquad\text{so}\qquad |F| \le 2q-1.
	\]
	If $q=1$ this second sub-case is vacuous, since $b_2<q-1=0$ is impossible for $b_2\ge0$. Assume $q\ge2$. If $(m\mathcal K_2)[B]$ has no perfect matching, that is $b_1\ge1$, and $b_2<q-1$, the same maximization with budget $q-2$ gives
	\[
	|B| \le m+q-2, \qquad\text{so}\qquad |F| \le m+q-1,
	\]
	tying the bound from the first case. Since $q\le m$ gives $2q-1\le m+q-1$, neither sub-case exceeds it.

	Combining both cases, $\dimn = m+q-1$.
\end{proof}

\begin{theorem}\label{thm:friendtop}
	For every $m\ge1$, with $q=\nu(\mathcal K_1\ast(m\mathcal K_2))=m$: $I(\mathcal K_1\ast(m\mathcal K_2))^{[q]}=\mathfrak m^{[2q]}$, so $R/I(\mathcal K_1\ast(m\mathcal K_2))^{[q]}$ is Cohen-Macaulay with $\dimn=\depth=\reg=2m-1$ and $\pd=2$.
\end{theorem}
\begin{proof}
	The graph has
	\[
	|V(\mathcal K_1\ast(m\mathcal K_2))| = 2m+1,
	\]
	odd for every $m$, so it never has a perfect matching itself, consistent with $\nu=m<(2m+1)/2$. We check instead that every single-vertex deletion satisfies condition (c) of Theorem \ref{thm:fm}.

	Deleting $v$ leaves $m\mathcal K_2$, which is its own perfect matching by Lemma \ref{lem:friendbasic}.

	Deleting any other vertex, say an endpoint of edge $i$, leaves $v$, the $m-1$ untouched edges, and the other endpoint of the deleted edge $i$. Match $v$ to this other endpoint of the deleted edge, and match each of the other $m-1$ edges to itself, this uses \(1+(m-1) = m\) edges, covering all $2m$ remaining vertices exactly. By the symmetry of the $m$ triangles about $v$, the same construction works for every such vertex.

	So condition (c) holds. By Theorem \ref{thm:fm}(a)$\Leftrightarrow$(b),
	\[
	I(\mathcal K_1\ast(m\mathcal K_2))^{[m]} = \mathfrak m^{[2m]},
	\]
	and Theorem \ref{thm:veronese}, applied in the $2m+1$ variables of the graph, gives $\dimn=\depth=\reg=2m-1$. The Auslander-Buchsbaum formula then gives
	\[
	\pd = (2m+1)-(2m-1) = 2. \qedhere
	\]
\end{proof}

\section{Concluding remarks}\label{sec:concl}

The depth of every squarefree power of every wheel graph is now completely determined: $\depth(R/I(\wn)^{[q]})=2q-1$ for every $n\ge3$ and every $1\le q\le\nu(\wn)$, obtained by combining the cone-depth reduction of Proposition \ref{prop:reduction} with a general lower bound on the depth of a squarefree monomial ideal in terms of its generator degrees  and the explicit facet argument of Theorem \ref{thm:jbound}. Two questions remain open. A closed formula for $\reg(R/I(\wn)^{[q]})$ is known only at the two endpoints $q=1$ and $q=\nu(\wn)$, for $1<q<\nu(\wn)$ it remains open. A full Krull dimension formula for the generalized wheel graphs $\cn\ast\mathcal K_m$ and $\cn\ast\overline{\mathcal K_m}$ requires a two-parameter extremal optimization we have not carried out. Beyond these, Theorem \ref{thm:joinmatching} and its consequences hold for arbitrary graphs, and should apply just as directly to any further family expressible as a join.

\bigskip
\noindent
\textbf{Data availability.}

\noindent\textbf{Conflict of interest.} The authors declare that they have no conflict of interest.


\begin{thebibliography}{99}

\bibitem{Berge58} C. Berge, Sur le couplage maximum d'un graphe, \textit{C. R. Acad. Sci. Paris} \textbf{247} (1958), 258--259.

\bibitem{BHN} M. Bigdeli, J. Herzog and R. Zaare-Nahandi, On the index of powers of edge ideals, \textit{Comm. Algebra} \textbf{46}(3) (2018), 1080--1095.

\bibitem{CFL} M. Crupi, A. Ficarra and E. Lax, Matchings, squarefree powers, and Betti splittings, \textit{Illinois J. Math.} \textbf{69}(2) (2025), 353--372.

\bibitem{DGS} S. Das, A. Ghosh and S. Selvaraja, Regularity of squarefree powers of edge ideals of whiskered cycles, preprint, arXiv:2604.17100.

\bibitem{DRS} K. K. Das, A. Roy and K. Saha, Square-free powers of Cohen-Macaulay forests, cycles, and whiskered cycles, preprint, arXiv:2409.06021.

\bibitem{EFMM} E. Emtander, R. Fr\"oberg, F. Mohammadi and S. Moradi, Poincar\'e series of some hypergraph algebras, preprint, arXiv:0901.1534.

\bibitem{EHHM} N. Erey, J. Herzog, T. Hibi and S. Saeedi Madani, Matchings and squarefree powers of edge ideals, \textit{J. Combin. Theory Ser. A} \textbf{188} (2022), 105585.

\bibitem{EHHM2} N. Erey, J. Herzog, T. Hibi and S. Saeedi Madani, The normalized depth function of squarefree powers, \textit{Collect. Math.} \textbf{75} (2024), 409--423.

\bibitem{Fak} S. A. Seyed Fakhari, On the Castelnuovo-Mumford regularity of squarefree powers of edge ideals, \textit{J. Pure Appl. Algebra} \textbf{228}(3) (2024).

\bibitem{FM} A. Ficarra and S. Moradi, Monomial ideals whose all matching powers are Cohen-Macaulay, arXiv:2410.01666.

\bibitem{herhib} W. Bruns and H. Herzog, \textit{Cohen-Macaulay Rings}, 2nd ed., Cambridge Studies in Advanced Mathematics, Cambridge University Press, 1998.

\bibitem{HH} J. Herzog and T. Hibi, Cohen-Macaulay polymatroidal ideals, \textit{European J. Combin.} \textbf{27}(4) (2006), 513--517.

\bibitem{HVT1} H. T. H\`a and A. Van Tuyl, Monomial ideals, edge ideals of hypergraphs, and their graded Betti numbers, \textit{J. Algebraic Combin.} \textbf{27}(2) (2008), 215--245.

\bibitem{hoch1} M. Hochster, Cohen-Macaulay rings, combinatorics, and simplicial complexes, in \textit{Ring Theory II} (Proc. Second Conf., Univ. Oklahoma, Norman, OK, 1975), 171--223.

\bibitem{jacq} S. Jacques, \textit{Betti Numbers of Graph Ideals}, Ph.D. thesis, University of Sheffield, 2004.

\bibitem{katz} M. Katzman, Characteristic-independence of Betti numbers of graph ideals, \textit{J. Combin. Theory Ser. A} \textbf{113} (2006), 435--454.

\bibitem{Mous} A. Mousivand, Algebraic properties of product of graphs, \textit{Comm. Algebra} \textbf{40}(11) (2012), 4177--4194.

\bibitem{peeva} I. Peeva, \textit{Graded Syzygies}, Algebra and Applications 14, Springer-Verlag, London, 2011.

\bibitem{rova} M. Roth and A. Van Tuyl, On the linear strand of an edge ideal, \textit{Comm. Algebra} \textbf{35} (2007), 821--832.

\bibitem{RPW} S. A. Rather, S. Pirzada and B. A. Wani, Certain homological invariants of wheel graphs, preprint, 2026.

\bibitem{Tutte47} W. T. Tutte, The factorization of linear graphs, \textit{J. London Math. Soc.} \textbf{22} (1947), 107--111.

\bibitem{V1} R. H. Villarreal, Cohen-Macaulay graphs, \textit{Manuscripta Math.} \textbf{66}(3) (1990), 277--293.

\end{thebibliography}
\end{document}